\documentclass[a4paper]{article}

\usepackage{amsmath,amsfonts,amssymb,amsthm}
\usepackage{thmtools,thm-restate}
\usepackage{pstricks,caption}

\newtheorem{lemma}{Lemma}[section]
\newtheorem{corollary}[lemma]{Corollary}
\newtheorem{theorem}[lemma]{Theorem}
\newtheorem{proposition}[lemma]{Proposition}

\theoremstyle{definition}
\newtheorem{definition}[lemma]{Definition}
\newtheorem{remark}[lemma]{Remark}

\newcommand{\Sphere}{\mathbb{S}^3}  %S^3
\newcommand{\sphere}{\mathbb{S}^2}  %S^2
\newcommand{\crcle}{\mathbb{S}^1}  %S^1
\DeclareMathOperator*{\nhd}{\mathcal{N}} %neighbourhood
\DeclareMathOperator{\Int}{int} %interior
\DeclareMathOperator*{\homeo}{Homeo} %group of homeomorphisms
\DeclareMathOperator{\Mod}{Mod} %mapping class group
\DeclareMathOperator{\im}{im} %image
\DeclareMathOperator{\SL}{SL} %special linear
\DeclareMathOperator{\tr}{tr} %trace

\begin{document}

\title{Point-pushing in $3$--manifolds}
\author{Jessica E. Banks}
\date{}%29 Oct 2013, 15 Apr 2014
\maketitle
\begin{abstract}
We study the Birman exact sequence for compact $3$--manifolds.
\end{abstract}

%------------------

\section{Introduction}
%master document is splitlinks.tex

The mapping class group $\Mod(M)$ of a manifold $M$ is the group of equivalence classes of self-homeomorphisms of $M$. A number of related definitions can be made by varying the type of self-homeomorphism considered and the equivalence relation used (see \cite{MR2850125} Section 2.1, for example).

In \cite{MR0243519}, Birman gave a relationship (now known as the `Birman exact sequence') between the mapping class group of a manifold $M$ and that of the same manifold with a finite number of points removed, focusing particularly on the case that $M$ is a surface. When only one puncture is added, the result is as follows (definitions are given below).

\begin{theorem}[\cite{MR2850125} Theorem 4.6; see also \cite{MR0243519} Theorem 1 and Corollary 1.3, and \cite{MR0375281} Theorem 4.2]\label{birmanthm}
Let $M$ be a compact, orientable surface with $\chi(M)<0$. Then there is an exact sequence
\[
1\to \pi_1(M,p)\to \Mod(M_p)\to\Mod(M)\to 1.
\]
\end{theorem}

The aim of this paper is to study this exact sequence for compact $3$--manifolds. Our main results are as follows.

\begin{restatable*}{theorem}{kernelthm}\label{kernelthm}
If $M$ is a compact, connected $3$--manifold then there is a 
finitely generated group $K$, a finite group $J$ and a group $L\leq\mathbb{Z}^3$ such that
there are exact sequences
\[1\to L\to K\to J\to 1\]
and
\[
1\to K\to\pi_1(M,p)\to\Mod(M_p)\to\Mod(M)\to 1.
\]
\end{restatable*}

\begin{restatable*}{corollary}{hypcor}\label{hyperboliccor}
Suppose $M$ is a compact, orientable $3$--manifold.
Then either there is an exact sequence
\[
1\to\pi_1(M,p)\to \Mod(M_p)\to \Mod(M)\to 1
\]
or $M$ is prime and Seifert fibred.
In particular, if $M$ is hyperbolic with finite volume 
then the exact sequence holds.
\end{restatable*}

Throughout we work in the piecewise linear category. Although Birman worked with homeomorphisms up to homotopy, we will consider them up to isotopy.
For surfaces, the following result of Epstein (extending work of Baer) shows that these two equivalence relations are the same for our purposes. That this is not true in general for $3$--manifolds is shown in \cite{MR836722}.

\begin{theorem}[\cite{MR0214087} Theorem 6.3]\label{surfacehomotoisothm}
Let $M$ be a compact, connected surface, and let $p\in\Int(M)$. Let $\psi\colon (M,\partial M)\to (M,\partial M)$ be a homeomorphism such that there exists a homotopy $H\colon(M\times[0,1],\partial M\times[0,1],p\times[0,1])\to(M,\partial M,p)$ from $\psi$ to the identity.
Then there is an isotopy
$H'\colon(M\times[0,1],\partial M\times[0,1],p\times[0,1])\to(M,\partial M,p)$ from $\psi$ to the identity.
\end{theorem}

Since the mapping class group has been well-studied, much of this material is well-known.
The intention is that this paper should be, on the whole, fairly self-contained.
In Sections \ref{pointpushingsection} and \ref{isotopysection} we give a number of basic results that will be needed for later proofs. In Section \ref{sequencessection} we draw these together to give the Birman exact sequence in the case of homeomorphisms that keep $\partial M$ fixed. We then seek to gain control of isotopies that may move $\partial M$, for fibred $3$--manifolds in Section \ref{fibredsection} and for non-fibred ones in Section \ref{nonfibredsection}.
Finally, in Section \ref{proofsection} we prove Theorem \ref{kernelthm} and Corollary \ref{hyperboliccor}.
Along the way we give more precise results for the exceptional cases in Corollary \ref{hyperboliccor}, as well as for non-orientable $3$--manifolds.

\smallskip

After the first release of this paper, Allen Hatcher pointed out an alternative approach to the proof of Theorem \ref{kernelthm}. Details of this can be found in \cite{2014arXiv1404.3689B}.

I wish to thank Steven Boyer for helpful conversations, and also the anonymous referee who suggested this project.

\section{Point-pushing in manifolds}\label{pointpushingsection}
%master document is splitlinks.tex

Let $M$ be a connected $n$--manifold. For the moment we will assume that $n\geq 3$; we will consider the case $n=2$ in Section \ref{fibredsection}. 
For a point $p\in M$, write $M_p=M\setminus\{p\}$. 
We will use this notation throughout this paper.

\begin{definition}\label{homeodefn}
Given a directed path $\rho\colon[0,1]\to \Int(M)$ with $\rho(0)=p$, choose an isotopy $H_{\rho}\colon (M\times[0,1],\partial M\times[0,1])\to(M,\partial M)$ such that, for $x\in M$ and $t\in[0,1]$,
\begin{align*}
H_{\rho}(x,0)=& x,\quad H_{\rho}(p,t)=\rho(t),\\
x\in\partial M &\Rightarrow H_{\rho}(x,t)=x.
\end{align*}
That is, $H_{\rho}$ is an ambient isotopy that takes the point $p$ once along the path $\rho$, keeping the boundary of $M$ fixed.
Define a homeomorphism $\phi_{\rho}\colon(M,\partial M)\to (M,\partial M)$ by $\phi_{\rho}(x)=H_{\rho}(x,1)$ for $x\in M$.
\end{definition}

This definition is dependent on the choice of isotopy $H_{\rho}$. The existence of at least one such isotopy 
is given by the following result.

\begin{lemma}[\cite{MR0163318} Corollary 2.3]
Let $H\colon M_1\times[0,1]\to M_2$ be an isotopy, where $M_1$ is a point and $M_2$ is any manifold. 
Suppose that $H$ is supported on an open set $\mathcal{N}\subseteq M_2$.
Then $H$ can be extended to an ambient isotopy supported on $\mathcal{N}$.
\end{lemma}

Our primary interest is the isotopy class of the map $\phi_{\rho}$, but at different points we will set different conditions on the isotopies in question. We will view $\phi_{\rho}$ both as a map from $M$ to $M$ and as a map from $M_{\rho(0)}$ to $M_{\rho(1)}$ without explicitly distinguishing between these.

\begin{lemma}\label{welldefinedlemma}
Let $p,p'\in\Int(M)$, and let $\rho\colon[0,1]\to\Int(M)$ be a path from $p$ to $p'$.
Let $H,H'\colon(M\times[0,1],\partial M\times[0,1])\to(M,\partial M)$ be isotopies beginning at the identity with $H(p,t)=H'(p,t)=\rho(t)$ for $t\in[0,1]$.
Then $x\mapsto H(x,1)$ is isotopic on $M_{p'}$ to $x\mapsto H'(x,1)$.
Furthermore, if $H$ and $H'$ both fix $\partial M$ throughout then the isotopy between the maps can also be made to do so.
\end{lemma}
\begin{proof}
For $t\in[0,1]$ let $h_t\colon M\to M$ be given by $h_t(x)=H(x,t)$ for $x\in M$.
Define $\hat{H}\colon (M\times[0,1],\partial M\times[0,1])\to(M,\partial M)$ by
$\hat{H}(x,t)=h_1(h_t^{-1}(H'(x,t)))$.
Then, for $x\in M$ and $t\in[0,1]$, 
\begin{align*}
\hat{H}(x,0)&=h_1(h_0^{-1}(H'(x,0)))=h_1(x),\\
\hat{H}(x,1)&=h_1(h_1^{-1}(H'(x,1)))=H'(x,1),\\
\hat{H}(p,t)&=h_1(h_t^{-1}(H'(p,t)))=h_1(h_t^{-1}(\rho(t)))=h_1(p)=p'.
\end{align*}
In addition, for $t\in[0,1]$, the map $x\mapsto\hat{H}(x,t)$ is a composition of homeomorphisms of $M$, and so is itself a homeomorphism of $M$.
Therefore $\hat{H}$ is an isotopy between the maps $x\mapsto H(x,1)$ and $x\mapsto H'(x,1)$ in $M_{p'}$.
Note that if $H$ and $H'$ fix the boundary throughout then so does $\hat{H}$.
\end{proof}

Thus the isotopy class of $\phi_{\rho}$ as a homeomorphism from $M_p$ to $M_{p'}$ is well-defined given the path $\rho$.

\begin{corollary}\label{inkernelcor}
Let $p,p'\in\Int(M)$, and let $\rho\colon[0,1]\to\Int(M)$ be a path from $p$ to $p'$.
Then $\phi_{\rho}$ is isotopic to the identity map on $M$ by an isotopy keeping $\partial M$ fixed.
\end{corollary}

\begin{corollary}\label{fixpointcor}
Let $\psi\colon(M,\partial M)\to(M,\partial M)$ be a homeomorphism.
Let $p\in\Int(M)$ and set $p'=\psi(p)$. Suppose $p'\neq p$.
Choose a path $\rho\colon[0,1]\to\Int(M)$ with $\rho(0)=p'$ and $\rho(1)=p$.
Then $\phi_{\rho}\circ\psi$ fixes $p$ and is isotopic to $\psi$ through maps of $M$ that keep $\partial M$ fixed.
\end{corollary}

\begin{remark}\label{reparamremark}
Let $\rho\colon[0,1]\to \Int(M)$ be a path, and let $H_{\rho}\colon(M\times[0,1],\partial M\times[0,1])\to(M,\partial M)$ be an isotopy that can be used to define $\phi_{\rho}$. Let $f\colon[0,1]\to[0,1]$ be an increasing bijection, and define $\rho'\colon[0,1]\to\Int(M)$ by $\rho'(t)=\rho(f(t))$.
Then $(x,t)\mapsto H_{\rho}(x,f(t))$ is an ambient isotopy that can be used to define $\phi_{\rho'}$.
Note that $\phi_{\rho}(x)=H_{\rho}(x,1)=H_{\rho}(x,f(1))=H_{\rho'}(x,1)=\phi_{\rho'}(x)$ for $x\in M$.
That is, re-parametrising the path $\rho$ does not change the isotopy class of $\phi_{\rho}$ on $M_{p'}$.
\end{remark}

The following theorem is a specific case of the central result of \cite{MR0310898} (see also \cite{MR0232394}).

\begin{theorem}\label{homotoisothm}
Let $M'$ be a compact, connected $n'$--manifold for $n'\geq 3$.
Let $\rho_1,\rho_2\colon [0,1]\to\Int(M')$ be 
embeddings with $\rho_1(0)=\rho_2(0)$ and $\rho_1(1)=\rho_2(1)$.
Suppose that $\rho_1$ and $\rho_2$ are homotopic relative to their endpoints. 
Then there is an ambient isotopy between them that keeps $\rho_1(0)$, $\rho_1(1)$ and $\partial M'$ fixed.
\end{theorem}

\begin{proposition}\label{homotopearcprop}
Let $\rho$ and $\rho'$ be paths in $\Int(M)$ with $\rho(0)=\rho'(0)=p$ and $\rho(1)=\rho'(1)=p'$.
Suppose that $\rho$ and $\rho'$ are both embedded, and that they are homotopic in $M$ relative to their endpoints.
Then $\phi_{\rho}$ and $\phi_{\rho'}$ are isotopic as maps from $M_p$ to $M_{p'}$ by an isotopy that keeps $\partial M$ fixed.
\end{proposition}
\begin{proof}
Theorem \ref{homotoisothm} gives an isotopy $H\colon (M\times[0,1],\partial M\times[0,1])\to(M,\partial M)$  that fixes $\partial M$, such that for $x\in M$ and $t\in[0,1]$,
\begin{align*}
H(x,0)&=x,\\
H(p,t)=p,\ \  & H(p',t)=p',\\
H(x,1)\in\rho' &\Leftrightarrow x\in\rho.
\end{align*}
By re-parametrising $\rho'$ if needed, we may assume $\rho'(t)=H(\rho(t),1)$ for $t\in[0,1]$.
Given $t\in[0,1]$, denote by $h_t$ the homeomorphism $x\mapsto H(x,t)$.

Let $H_{\rho}\colon(M\times[0,1],\partial M\times[0,1])\to(M,\partial M)$ be the isotopy used to define $\phi_{\rho}$.
Define a third isotopy $H_{\rho'}\colon(M\times[0,1],\partial M\times[0,1])\to(M,\partial M)$ by $H_{\rho'}(x,t)=h_1(H_{\rho}(h_1^{-1}(x),t))$. 
Then for $x\in M$ and $t\in[0,1]$,
\begin{align*}
H_{\rho'}(x,0)&=h_1(H_{\rho}(h_1^{-1}(x),0))=h_1(h_1^{-1}(x))=x,\\
H_{\rho'}(p,t)&=h_1(H_{\rho}(h_1^{-1}(p),t))=h_1(H_{\rho}(p,t))=H(\rho(t),1)=\rho'(t),\\
x\in\partial M &\Rightarrow H_{\rho'}(x,t)=h_1(H_{\rho}(h_1^{-1}(x),t))=x.
\end{align*}
Hence $H_{\rho'}$ can be used to define $\phi_{\rho'}$.

Define a fourth isotopy $\hat{H}\colon(M\times[0,1],\partial M\times[0,1])\to(M,\partial M)$ by $\hat{H}(x,t)=h_t(\phi_{\rho}(h_t^{-1}(x)))$.
Then for $x\in M$ and $t\in[0,1]$,
\begin{align*}
\hat{H}(x,0)&=h_0(\phi_{\rho}(h_0^{-1}(x)))
=\phi_{\rho}(x),\\
\hat{H}(x,1)&=h_1(H_{\rho}(h_1^{-1}(x),1))=H_{\rho'}(x,1)=\phi_{\rho'}(x),\\
\hat{H}(p,t)&=h_t(\phi_{\rho}(h_t^{-1}(p)))=h_t(\phi_{\rho}(p))=h_t(p')=p'.
\end{align*}
Therefore $\hat{H}$ is an isotopy between $\phi_{\rho}$ and $\phi_{\rho'}$ that sends $p$ to $p'$ throughout. 
Moreover, $\hat{H}$ keeps $\partial M$ fixed.
\end{proof}

We now extend Proposition \ref{homotopearcprop} to all paths rather than just embedded ones. 
In doing this, it is important that we are working in the piecewise linear category, as this implies that all paths (except the constant paths) are piecewise embedded up to re-parametrisation.

\begin{corollary}\label{pwembeddedcor}
Let $\rho$ and $\rho'$ be paths in $\Int(M)$ with $\rho(0)=\rho'(0)=p$ and $\rho(1)=\rho'(1)=p'$.
Suppose that $\rho$ and $\rho'$ 
are homotopic in $M$ keeping their endpoints fixed.
Then $\phi_{\rho}$ and $\phi_{\rho'}$ are isotopic as maps from $M_p$ to $M_{p'}$ by an isotopy that keeps $\partial M$ fixed.
\end{corollary}
\begin{proof}
First assume that $p\neq p'$.
We will show that there exists an embedded path $\rho''$ homotopic to $\rho$ relative to its endpoints such that $\phi_{\rho}$ is isotopic to $\phi_{\rho''}$ through maps fixing $p'$ and $\partial M$. It will then follow that an analogous such path exists for $\rho'$, and applying Proposition \ref{homotopearcprop} to these two paths will complete the proof.

Choose $t_0=0<t_1<\cdots<t_{n-1}<t_n=1$ that divide $\rho$ into a finite sequence $\rho_1,\ldots,\rho_n$ of embedded paths.
Isotope each path $\rho_i$ keeping its endpoints fixed to a path $\rho_{i,a}$ such that $\rho_{i,a}$ is embedded and, away from its endpoints, is disjoint from $\rho$ and from $\rho_{j,a}$ for $j\neq i$.
Figure \ref{pic1} shows a schematic picture of this in the case $n=3$, with the path $\rho$ taking the form $\sigma\cdot\sigma^{-1}\cdot\sigma$ for an embedded arc $\sigma$.
Let $\rho_a$ be the path $\rho_{1,a}\cdot\rho_{2,a}\cdots\rho_{n,a}$.
By Proposition \ref{homotopearcprop}, 
$\phi_{\rho_{1}}$ is isotopic to $\phi_{\rho_{1,a}}$ in $M_{\rho(t_1)}$,
and $\phi_{\rho_a}%=\phi_{\rho_{1,a}\cdot\rho_{2,a}\cdots\rho_{n,a}}(R)
=\phi_{\rho_{n,a}}\circ\cdots \circ\phi_{\rho_{1,a}}$ is isotopic to $\phi_{\rho_{n}}\circ\cdots\circ \phi_{\rho_{1}}=\phi_{\rho}$ in $M_{\rho(t_n)}=M_{p'}$.
All these isotopies also fix $\partial M$.

Now, for $1\leq i\leq n$, let $t_{i,a}=\frac{1}{2}(t_{i-1}+t_i)$.
Then $\rho_a(t_{1,a}),\ldots,\rho_a(t_{n,a})$ are distinct from each other and from $p$ and $p'$. Together with $t_0$ and $t_n$, these values of $t$ divide $\rho_a$ into paths $\rho_{0,b},\rho_{1,b},\ldots,\rho_{n,b}$, each of which is embedded.
Isotope each path $\rho_{i,b}$ keeping its endpoints fixed to a path $\rho_{i,c}$ so that the path $\rho_c=\rho_{0,c}\cdots\rho_{n,c}$ is  embedded (again, see Figure \ref{pic1}).
Then, again by Proposition \ref{homotopearcprop}, $\phi_{\rho_c}$ is isotopic to $\phi_{\rho_a}$ in $M_{p'}$, and so is also isotopic to $\phi_{\rho}$.
Thus $\rho_c$ is the path $\rho''$ required.
\begin{figure}[htbp]
\centering
%LaTeX with PSTricks extensions
%%Creator: 0.48.2
%%Please note this file requires PSTricks extensions
\psset{xunit=.5pt,yunit=.5pt,runit=.5pt}
\begin{pspicture}(540,320)
{
\pscustom[linewidth=3,linecolor=black]%thick line
{
\newpath
\moveto(120,40)
\lineto(120,280)
}
}
{
\pscustom[linewidth=1,linecolor=black]%thinner lines
{
\newpath
\moveto(120,40)
\curveto(20,100)(20.000692,140)(20,180)
\curveto(19.999227,224.72093)(60,260)(120,280)
\curveto(180,260)(200,200)(200,160)
\curveto(200,120)(139.97631,40.97312)(120,40)
\curveto(98.97583,38.97583)(69.87797,145.15661)(80,200)
\curveto(83.62994,219.66783)(120,280)(120,280)
\moveto(330,170)
\curveto(310,270)(350,310)(430,310)
\curveto(510,310)(530,264.72136)(530,220)
\curveto(530,220)(530,170)(510,170)
\curveto(460,150)(470,150)(390,150)
\moveto(430,40)
\curveto(330,100)(330.00069,140)(330,180)
\curveto(329.99923,224.72093)(370,260)(430,280)
\curveto(490,260)(510,200)(510,160)
\curveto(510,120)(449.97631,40.97312)(430,40)
\curveto(408.97583,38.97583)(379.87797,145.15661)(390,200)
\curveto(393.62994,219.66783)(430,280)(430,280)
}
}
{
\pscustom[linestyle=none,fillstyle=solid,fillcolor=black]%circle
{
\newpath
\moveto(130,39.99999738)
\curveto(130,34.47714988)(125.5228475,29.99999738)(120,29.99999738)
\curveto(114.4771525,29.99999738)(110,34.47714988)(110,39.99999738)
\curveto(110,45.52284488)(114.4771525,49.99999738)(120,49.99999738)
\curveto(125.5228475,49.99999738)(130,45.52284488)(130,39.99999738)
\closepath
\moveto(130,279.99999738)
\curveto(130,274.47714988)(125.5228475,269.99999738)(120,269.99999738)
\curveto(114.4771525,269.99999738)(110,274.47714988)(110,279.99999738)
\curveto(110,285.52284488)(114.4771525,289.99999738)(120,289.99999738)
\curveto(125.5228475,289.99999738)(130,285.52284488)(130,279.99999738)
\closepath
\moveto(440,39.99999738)
\curveto(440,34.47714988)(435.5228475,29.99999738)(430,29.99999738)
\curveto(424.4771525,29.99999738)(420,34.47714988)(420,39.99999738)
\curveto(420,45.52284488)(424.4771525,49.99999738)(430,49.99999738)
\curveto(435.5228475,49.99999738)(440,45.52284488)(440,39.99999738)
\closepath
\moveto(440,279.99999738)
\curveto(440,274.47714988)(435.5228475,269.99999738)(430,269.99999738)
\curveto(424.4771525,269.99999738)(420,274.47714988)(420,279.99999738)
\curveto(420,285.52284488)(424.4771525,289.99999738)(430,289.99999738)
\curveto(435.5228475,289.99999738)(440,285.52284488)(440,279.99999738)
\closepath
\moveto(395,149.99999738)
\curveto(395,147.23857363)(392.76142375,144.99999738)(390,144.99999738)
\curveto(387.23857625,144.99999738)(385,147.23857363)(385,149.99999738)
\curveto(385,152.76142113)(387.23857625,154.99999738)(390,154.99999738)
\curveto(392.76142375,154.99999738)(395,152.76142113)(395,149.99999738)
\closepath
\moveto(335,169.99999738)
\curveto(335,167.23857363)(332.76142375,164.99999738)(330,164.99999738)
\curveto(327.23857625,164.99999738)(325,167.23857363)(325,169.99999738)
\curveto(325,172.76142113)(327.23857625,174.99999738)(330,174.99999738)
\curveto(332.76142375,174.99999738)(335,172.76142113)(335,169.99999738)
\closepath
\moveto(515,169.99999738)
\curveto(515,167.23857363)(512.76142375,164.99999738)(510,164.99999738)
\curveto(507.23857625,164.99999738)(505,167.23857363)(505,169.99999738)
\curveto(505,172.76142113)(507.23857625,174.99999738)(510,174.99999738)
\curveto(512.76142375,174.99999738)(515,172.76142113)(515,169.99999738)
\closepath
}
}
{
\pscustom[linestyle=none,fillstyle=solid,fillcolor=black]%arrowhead
{
\newpath
\moveto(200,170)
\lineto(196,174)
\lineto(200,160)
\lineto(204,174)
\lineto(200,170)
\closepath
\moveto(20,170)
\lineto(24,166)
\lineto(20,180)
\lineto(16,166)
\lineto(20,170)
\closepath
\moveto(78,170)
\lineto(82,166)
\lineto(78,180)
\lineto(74,166)
\lineto(78,170)
\closepath
\moveto(120,150)
\lineto(124,146)
\lineto(120,160)
\lineto(116,146)
\lineto(120,150)
\closepath
\moveto(120,210)
\lineto(124,206)
\lineto(120,220)
\lineto(116,206)
\lineto(120,210)
\closepath
\moveto(120,96)
\lineto(116,100)
\lineto(120,86)
\lineto(124,100)
\lineto(120,96)
\closepath
\moveto(353.29998877,100.52001797)
\lineto(358.81197709,99.24802067)
\lineto(348,109)
\lineto(352.02799146,95.00802965)
\lineto(353.29998877,100.52001797)
\closepath
\moveto(425,310)
\lineto(421,306)
\lineto(435,310)
\lineto(421,314)
\lineto(425,310)
\closepath
\moveto(420,150)
\lineto(424,154)
\lineto(410,150)
\lineto(424,146)
\lineto(420,150)
\closepath
\moveto(388,172)
\lineto(392,168)
\lineto(388,182)
\lineto(384,168)
\lineto(388,172)
\closepath
}
}
{
\put(420,320){$\rho_{1,c}$}
\put(105,300){$p'=\rho(t_1)=\rho(t_3)$}
\put(105,10){$p=\rho(t_0)=\rho(t_2)$}
\put(450,140){$\rho_{2,c}$}
\put(320,80){$\rho_{0,b}$}
\put(400,200){$\rho_{3,b}$}
\put(5,90){$\rho_{1,a}$}
\put(50,210){$\rho_{3,a}$}
\put(130,200){$\rho_1$}
\put(130,90){$\rho_2$}
\put(130,150){$\rho_3$}
\put(160,140){$\rho_{2,a}$}
\put(260,165){$\rho_{a}(t_{1,a})$}
\put(520,165){$\rho_{a}(t_{2,a})$}
\put(395,160){$\rho_{a}(t_{3,a})$}
}
\end{pspicture}
\caption{\label{pic1}}
\end{figure}

Now suppose that $p=p'$.
Choose an embedded path $\sigma$ from $p$ to a distinct point $p''$.
Then the paths $\rho\cdot\sigma$ and $\rho'\cdot\sigma$ are piecewise embedded and are homotopic keeping their endpoints fixed. Thus $\phi_{\rho\cdot\sigma}$ is isotopic to $\phi_{\rho'\cdot\sigma}$ in $M_{p''}$.
Accordingly, $\phi_{\sigma}^{-1}\circ\phi_{\rho\cdot\sigma}$ and $\phi_{\sigma}^{-1}\circ\phi_{\rho'\cdot\sigma}$ are isotopic in $M_{p'}$.
This completes the proof, since $[\phi_{\sigma^{-1}}(\phi_{\rho\cdot\sigma})]=[\phi_{\sigma}^{-1}(\phi_{\sigma}(\phi_{\rho}))]=[\phi_{\rho}]$ and similarly $[\phi_{\sigma^{-1}}(\phi_{\rho'\cdot\sigma})]=[\phi_{\rho'}]$.
\end{proof}

\begin{remark}
Although the constant path $t\mapsto p$ is not piecewise embedded, the proof of Corollary \ref{pwembeddedcor} also holds in this case.
For convenience, we will treat the constant path as a simple closed curve. The reader may, if they prefer, replace this with an embedded loop within a neighbourhood of $p$.
\end{remark}

\begin{corollary}\label{makesimplecor}
Let $p\in\Int(M)$ and let $\rho$ be a 
loop in $\Int(M)$ based at $p$.
Then there exists a simple closed curve $\rho'$ in $\Int(M)$ based at $p$ such that $\phi_{\rho}$ and $\phi_{\rho'}$ are isotopic maps of $M_p$.
\end{corollary}

\begin{lemma}[\cite{MR0163318} Addendum 2.2]\label{bdyisotopylemma}
Let $M'$ be a manifold with compact boundary, and let $\nhd(\partial M')$ be a neighbourhood of $\partial M'$.
Then any isotopy of $\partial M'$ can be extended to an ambient isotopy supported on $\nhd(\partial M')$.
\end{lemma}

\begin{lemma}\label{replaceboundarylemma}
Let $\psi\colon (M,\partial M)\to (M,\partial M)$ be a homeomorphism that is isotopic to the identity on $M$.
Then there exists a homeomorphism $\psi'\colon (M,\partial M)\to (M,\partial M)$ that is isotopic to the identity on $M$ by an isotopy that keeps $\partial M$ fixed, such that $\psi'$ is isotopic to $\psi$ by an isotopy that keeps $p$ fixed.
\end{lemma}
\begin{proof}
Let $H\colon (M\times[0,1],\partial M\times[0,1])\to(M,\partial M)$ be the isotopy from the identity to $\psi$.
Choose a neighbourhood $\nhd(\partial M)$ of $\partial M$ that does not contain $p$.
By Lemma \ref{bdyisotopylemma}, the restriction of $H$ to $\partial M$ extends to an ambient isotopy 
$H'\colon(M\times[0,1],\partial M\times[0,1])\to(M,\partial M)$ that is supported on $\nhd(\partial M)$.
For $t\in[0,1]$, denote by $h'_t$ the map $x\mapsto H'(x,t)$.

Let $\psi'=(h'_1)^{-1}\circ\psi$. Then $\psi$ is isotopic to $\psi'$ by the isotopy $(x,t)\mapsto (h'_t)^{-1}(\psi(x))$. Since $H'$ is supported on $\nhd(\partial M)$, this isotopy fixes $p$.
Now define a further isotopy $\hat{H}\colon (M\times[0,1],\partial M\times[0,1])\to(M,\partial M)$ by
$\hat{H}(x,t)=(h'_t)^{-1}(H(x,t))$.
Then for $x\in M$ and $t\in[0,1]$,
\begin{align*}
\hat{H}(x,0)&=(h'_0)^{-1}(H(x,0))=x,\\
\hat{H}(x,1)&=(h'_1)^{-1}(H(x,1))=(h'_1)^{-1}(\psi(x))=\psi'(x),\\
x\in\partial M&\Rightarrow \hat{H}(x,t)=(h'_t)^{-1}(H'(x,t))=x.
\end{align*}
Hence $\psi'$ is isotopic to the identity by an isotopy that fixes $\partial M$.
\end{proof}

\begin{lemma}\label{inimagelemma}
Let $p\in\Int(M)$, and let $\psi\colon(M,\partial M,p)\to(M,\partial M,p)$ be a homeomorphism that is isotopic to the identity on $M$.
Then there exists a simple closed curve $\rho$ in $\Int(M)$ such that $\psi$ is isotopic to $\phi_{\rho}$ keeping $p$ fixed.

Moreover, if the isotopy between $\psi$ and the identity fixes $\partial M$ then the isotopy between $\psi$ and $\phi_{\rho}$ can also be made to do so.
\end{lemma}
\begin{proof}
Let $H\colon (M\times[0,1],\partial M\times[0,1])\to(M,\partial M)$ be the isotopy from the identity to $\psi$.

If $H$ fixes $\partial M$, define $\rho'\colon[0,1]\to\Int(M)$ by $\rho(t)=H(p,t)$.
Then $H$ can be used to define $\phi_{\rho'}$, so $\phi_{\rho'}=\psi$. 
Using Corollary \ref{makesimplecor} we may replace $\rho'$ by a simple closed curve $\rho$.

If $H$ does not fix $\partial M$, by Lemma \ref{replaceboundarylemma} we may first isotope $\psi$, keeping $p$ fixed, to a homeomorphism $\psi'$ of $M_p$ that is isotopic to the identity by an isotopy that fixes $\partial M$.
Then there exists a simple closed curve $\rho$ based at $p$ such that $\psi'$ is isotopic to $\phi_{\rho}$ keeping $p$ fixed.
\end{proof}

\section{Altering isotopies}\label{isotopysection}
%master document is splitlinks.tex

The main purpose of this section is to prove Lemma \ref{fixboundarylemma}, which under certain conditions gives control over the boundary of a surface being isotoped within a $3$--manifold.
We will not make use of this lemma until Section \ref{nonfibredsection}, but we include the proof here as it is closer in style to those of Section \ref{pointpushingsection}.

\begin{theorem}[\cite{MR0214087} Theorem 5.2 (The Alexander Trick)]\label{alextrickthm}
Let $M'$ be an $m$--ball for some $m\in\mathbb{Z}_{\geq 0}$, and let $p'\in \Int(M')$.
Let $\psi\colon(M',p')\to(M',p')$ be a homeomorphism that restricts to the identity on $\partial M'$.
Then $\psi$ is isotopic to the identity on $M'$ by an isotopy that keeps $\partial M'$ and $p'$ fixed.
\end{theorem}

\begin{theorem}[\cite{MR0214087} Theorem 5.3]\label{spheremapthm}
Let $M'$ be an $m$--sphere for some $m\in\mathbb{Z}_{\geq 0}$, and let $p'\in M'$.
Let $\psi\colon(M',p)\to(M',p)$ be an orientation-preserving homeomorphism. Then $\psi$ is isotopic to the identity on $M'$ by an isotopy that keeps $p'$ fixed.
\end{theorem}

\begin{theorem}[\cite{MR0214087} Theorem 3.1]\label{homotopicarcsthm}
Let $M'$ be a compact surface, and let $\rho_1$ and $\rho_2$ be arcs properly embedded in $M'$ that are homotopic in $M'$ keeping their endpoints fixed. Then $\rho_1$ and $\rho_2$ are isotopic by an ambient isotopy that keeps $\partial M'$ fixed.
\end{theorem}

\begin{lemma}\label{isotopeannuluslemma}
Let $M'=\crcle\times[0,1]$ be an annulus, and $\psi\colon (M',\partial M')\to (M',\partial M')$ a homeomorphism with $\psi(x,0)=(x,0)$ for $x\in\crcle$. Then there is an isotopy $H\colon (M'\times[0,1],\partial M'\times[0,1])\to (M',\partial M')$ from $\psi$ to the identity on $M'$ with $H(x,0,t)=(x,0)$ for $x\in\crcle$ and $t\in[0,1]$.
\end{lemma}
\begin{proof}
Let $X=\mathbb{R}/\sim$, where $x\sim x+m$ for $x\in\mathbb{R}$ and $m\in\mathbb{Z}$, so that $X\cong\crcle$.
Then $M'\cong X\times[0,1]$.
Let $\widetilde{M}'=\mathbb{R}\times[0,1]$, so that $\widetilde{M}'$ is the universal cover of $M'$, and let $\pi_{\mathbb{R}}\colon\widetilde{M}'\to\mathbb{R}$ be projection onto the first factor.

Let $\rho_0\colon[0,1]\to M'$ be the arc in $M'$ given by $\rho_0(t)=(0,t)$ for $t\in[0,1]$.
The homeomorphism $\psi$ maps $\rho_0$ to an arc $\rho_1\colon[0,1]\to M'$ with $\rho_1(0)=\rho_0(0)$.
Let $\widetilde{\rho}_1\colon[0,1]\to\widetilde{M}'$ be the lift of $\rho_1$ beginning at $(0,0)$.
Let $y\in\mathbb{R}$ such that $\widetilde{\rho}_1(1)=(y,1)$.
Define $H_{\mathbb{R}}\colon\mathbb{R}\times\{1\}\times[0,1]\to\mathbb{R}\times\{1\}$ by
$H_{\mathbb{R}}(x,1,t)=(x-yt,1)$ for $x\in\mathbb{R}$ and $t\in[0,1]$.
This descends to an isotopy of $X\times\{1\}\subseteq M'$. 
By Lemma \ref{bdyisotopylemma}, this extends to an isotopy $H_1\colon(M'\times[0,1],\partial M'\times[0,1])\to(M',\partial M')$ that keeps $X\times\{0\}$ fixed.

Let $\rho_2\colon[0,1]\to M'$ be the arc $\rho_2(t)=H_1(\rho_1(t),1)$, which is the image of $\rho_1$ under the isotopy $H_1$, and let $\widetilde{\rho}_2\colon[0,1]\to\widetilde{M}'$ be the lift of $\rho_2$ beginning at $\rho_2$.
As $\pi_{\mathbb{R}}(\widetilde{\rho}_2(1))=0$, the loop $\rho_2\cdot\rho_0^{-1}$ is null-homotopic in $M'$, meaning the arcs $\rho_0$ and $\rho_2$ are homotopic keeping their endpoints fixed.

Now consider the circle $X\times\{1\}$.
The map $(x,1)\mapsto H_1(\psi(x,1),1)$ is an orientation-preserving homeomorphism that takes $(0,1)$ to itself.
By Theorem \ref{spheremapthm}, this map is isotopic to the identity map on $X\times\{1\}$ by an isotopy keeping $(0,1)$ fixed.
Lemma \ref{bdyisotopylemma} shows that this isotopy can be achieved by an ambient isotopy of $M'$ that also keeps $X\times\{0\}$ fixed.

Although the arc $\rho_2$ may move under this isotopy, $\rho_0$ and $\rho_2$ are still isotopic relative to their endpoints,
so by Theorem \ref{homotopicarcsthm} there is an ambient isotopy of $M'$, fixed on $\partial M'$, after which $\rho_2$ coincides with $\rho_0$. 
We now have that $X\times\{0,1\}\cup \{0\}\times[0,1]$ is mapped by the identity. The complement of this set in $M'$ is an open disc. By Theorem \ref{alextrickthm}, there is therefore a final isotopy, fixed on $X\times\{0,1\}\cup \{0\}\times[0,1]$, to the identity.
\end{proof}

\begin{lemma}\label{fixboundarylemma}
Suppose that $M$ is a $3$--manifold and 
let $S$ be a toral boundary component of $M$.
Let $R\subseteq M$ be a properly embedded surface.
Suppose that every loop in $R\cap S$ is essential in $S$.
Then $R\cap S$ consists of finitely many parallel simple closed curves.

Choose an essential simple closed curve $\sigma$ on $S$, based at a point $p'$ on $S$, that is parallel to the curves of $R\cap S$ if any exist.
Define $\theta\colon\pi_1(S,p')\to \mathbb{Z}$ to be the algebraic intersection with $\sigma$, and let $G=\ker(\theta)$.
Let $(\widetilde{S},\widetilde{p}')$ be the infinite cyclic cover of $(S,p')$ with $\pi_1(\widetilde{S},\widetilde{p}')\cong G$.

Let $H\colon (M\times [0,1],\partial M\times[0,1])\to (M,\partial M)$ be an isotopy from the identity to a homeomorphism $\psi\colon (M,\partial M)\to (M,\partial M)$ where $\psi(R\cap S)=R\cap S$.
Let $H_S\colon S\times[0,1]\to S$ and $\psi_S\colon S\to S$ be the restrictions of $H$ and $\psi$ respectively to $S$.
Then $H_S$ lifts to an isotopy $\widetilde{H}_S\colon\widetilde{S}\times[0,1]\to\widetilde{S}$ from the identity on $\widetilde{S}$ to a lift $\widetilde{\psi}_S\colon\widetilde{S}\to\widetilde{S}$ of $\psi$.
Let $\widetilde{R}_S$ be a lift of $R\cap S$ to $\widetilde{S}$.

Suppose that $\widetilde{\psi}_S(\widetilde{R}_S)=\widetilde{R}_S$.
Then there is an isotopy $H'\colon (M\times [0,1],\partial M\times[0,1])\to (M,\partial M)$ from the identity on $M$ to $\psi$ such that $H'(R\cap S,t)=R\cap S$ for $t\in[0,1]$.
\end{lemma}
\begin{proof}
If $R\cap S=\emptyset$ then the result is trivial. We will therefore assume otherwise. 

Choose a product neighbourhood $\nhd(S)=S\times[0,2]$ of $S$ in $M$, with $S=S\times\{0\}$.
We may choose the product structure such that $R\cap\nhd(S)=(R\cap S)\times[0,1]$.
We will define $H'$ as a concatenation of three isotopies $H_1$, $H_2$ and $H_3$, where $H_2$ and $H_3$ are supported in $\nhd(S)$.

For $t\in[0,1]$, define a bijection $f_t\colon M\to M\setminus(S\times[0,t))$ by, for $x\in M$,
\[
f_t(x)=
\begin{cases}
x & x\in M\setminus\nhd(S),\\
(y,s) &x=(y,s)\in S\times[0,2],2t\leq s,\\
(y,\frac{s+2t}{2}) &x=(y,s)\in S\times[0,2], s<2t.

\end{cases}
\]
That is, $f_t$ takes $S\times[0,2]$ to $S\times[t,2]$.
For $y\in S$, $f_t(y,2)=(y,2)$, so $f_t$ is continuous.
Next define $H_1\colon (M\times[0,1],\partial M\times[0,1])\to (M,\partial M)$ by, for $x\in M$ and $t\in[0,1]$,
\[
H_1(x,t)=
\begin{cases}
f_1(H(f_1^{-1}(x),t)) & x\in M\setminus(S\times[0,1)),\\
(H_S(y,s),s)& x=(y,s)\in S\times[0,1), s\leq t,\\
(H_S(y,t),s)& x=(y,s)\in S\times[0,1), t<s.
\end{cases}
\]
For $y\in S$ and $t\in[0,1]$, 
\begin{align*}
H_1((y,1),t))&=f_1(H(f_1^{-1}(y,1),t))=f_1(H((y,0),t))\\
&=f_1(H_S(y,t),0)=(H_S(y,t),1),
\end{align*}
 so $H_1$ is continuous. 
Now $x\mapsto H_1(x,1)$ coincides with $\psi$ on $M\setminus\nhd(S)$ and
acts as a homeomorphism on $S\times\{1\}$.
Note also that, for $t\in[0,1]$, the map $x\mapsto H_1(x,t)$ restricts to the identity on $S$.

Define $\psi_{\partial}\colon S\times[0,1]\to S\times[0,1]$ by
$\psi_{\partial}(y,s)=(\psi_S(y),s)$ for $y\in S$ and $s\in[0,1]$.
We will define $H_2$ so that it is supported on $S\times[0,1]$ (and therefore keeps $S\times\{1\}$ fixed) and such that $H_2(x,1)=\psi_{\partial}(x)$ for $x\in S\times[0,1]$.

Finally, define $H_3\colon (M\times[0,1],\partial M\times[0,1])\to (M,\partial M)$ by,  for $x\in M$ and $t\in[0,1]$,
\[H_3(x,t)=
f_t^{-1}(H_2(f_t(x),1)).
\]
For $t\in[0,1]$, the condition that $H_2(x,1)=\psi_{\partial}(x)$ for $x\in S\times[0,1]$ ensures that $x\mapsto H_2(x,1)$ is a homeomorphism of $M\setminus (S\times[0,t))$, and so the map $x\mapsto H_3(x,t)$ is a homeomorphism of $M$. Therefore $H_3$ is an isotopy.
In addition, for each $t\in[0,1]$, for $x\in M$,
\begin{align*}
H_3(x,t)\in R\cap S &\Leftrightarrow f_t^{-1}(H_2(f_t(x),1))\in R\cap S\\
&\Leftrightarrow H_2(f_t(x),1)\in f_t(R\cap S)\subseteq S\times[0,1]\\
&\Leftrightarrow \psi_{\partial}(f_t(x))\in (R\cap S)\times\{t\}\\
&\Leftrightarrow x=(y,s)\in S\times[0,1], \psi_S(y)\in R\cap S,\\
& \left(\Big(2t\leq s \textrm{ and } s=t\Big)\textrm{ or }\Big(t<s\textrm{ and }\frac{s+2t}{2}=t\Big)\right)\\
&\Leftrightarrow x=(y,s)\in S\times[0,1], y\in R\cap S, s=0\\
&\Leftrightarrow x\in R\cap S.
\end{align*}
Lastly, for $x\in M$, $f_1(x)\in M\setminus(S\times [0,1))$ so
$H_3(x,1)=f_1^{-1}(H_2(f_1(x),1))=f_1^{-1}(H_1(f_1(x),1))=H(x,1)=\psi(x)$.

\medskip

We now return to defining the isotopy $H_2$. Since it will be supported on $S\times[0,1]$, we will only give the definition there. 

We may assume that $p'\in R$ and $\sigma$ 
runs once around the component of $R\cap S$ containing $p'$.
Choose an essential simple closed curve $\sigma'$ on $S$, based at $p'$, that crosses each curve of $R\cap S$ exactly once.
Express $S$ as $([0,1]\times[0,1])/\!\sim_S$, where $(0,t)\sim_S(1,t)$ and $(t,0)\sim_S(t,1)$ for $t\in[0,1]$, such that $\sigma(t)=(t,0)$ and $\sigma'(t)=(0,t)$ for $t\in[0,1]$, and moreover $R\cap S=[0,1]\times(R\cap\sigma')$.
Lifting this product structure, express $\widetilde{S}$ as $([0,1]\times\mathbb{R})/\!\sim_{\widetilde{S}}$, where $(0,t)\sim_{\widetilde{S}}(1,t)$ for $t \in\mathbb{R}$.

Let $\widetilde{\sigma}$ be the lift of $\sigma$ in $\widetilde{R}_S$ beginning at $\widetilde{p}'$.
We may assume that $\widetilde{p}'=\widetilde{\sigma}(0)=(0,0)$, and hence also that $\widetilde{\sigma}(t)=(t,0)$ for $t\in[0,1]$.
Each component of $\widetilde{R}_S$ is separating in $\widetilde{S}$. Therefore, the isotopy $\widetilde{H}_S$ cannot change the order in which these components lie in $\widetilde{S}\cong\crcle\times\mathbb{R}$.
Since $\widetilde{\psi}_S(\widetilde{R}_S)=\widetilde{R}_S$, 
this implies that, considered as a set,
$\widetilde{\psi}_S(\widetilde{\sigma})=\widetilde{\sigma}$.

Let $T_{\sigma}$ be the annulus $\sigma\times[0,1]$ in $S\times[0,1]$, and $T_{\sigma'}$ the annulus $\sigma'\times[0,1]$.
Let $\rho$ be the arc of $T_{\sigma}\cap T_{\sigma'}$, which is given by $\rho(t)=(p',t)$. 
We will use $T_{\sigma}$ and $T_{\sigma'}$ to guide the definition of the isotopy $H_2$, using standard $3$--manifold techniques.
However, care is needed in this case as the purpose of Lemma \ref{fixboundarylemma} is to control the isotopy on $S\times[0,1]$.
First isotope $T_{\sigma}$, $T_{\sigma'}$ and $\rho$ using the isotopy $H_1$.

Let $\widetilde{\rho}$ be the lift of $\rho$ to $\widetilde{S}\times[0,1]$ beginning at $(0,0,0)$, which is given by $\widetilde{\rho}(t)=(\widetilde{\psi}_S(p'),t)$.
Let $\pi_{\mathbb{R}}\colon\widetilde{S}\times[0,1]\to\mathbb{R}$ be projection onto the second factor of $\widetilde{S}=([0,1]\times\mathbb{R})/\!\sim_{\widetilde{S}}$.

Define a map $g\colon[0,1]\times[0,1]\to\mathbb{R}$ by
considering $[0,1]\times[0,1]$ as a cone over the point $(\frac{1}{2},\frac{1}{2})$ and specifying that, for $s,t\in[0,1]$,
\begin{align*}
g(s,0)=0, &\quad g(s,1)=\pi_{\mathbb{R}}(\widetilde{\rho}(s)),\\
g(0,t)=0,&\quad g(1,t)=0,\\
g\Big(\frac{1}{2},\frac{1}{2}\Big)&=0.
\end{align*}
Define an isotopy $\widetilde{H}_{\rho}\colon\widetilde{S}\times[0,1]\times[0,1]\to \widetilde{S}\times[0,1]$ by $\widetilde{H}_{\rho}(x,y,s,t)=(x,y-g(s,t),s)$.
Then $\widetilde{H}_{\rho}$ descends to an isotopy $H_{\rho}\colon S\times[0,1]\times[0,1]\to S\times[0,1]$.
As $g(0,t)=g(1,t)=0$, the map $(x,y,s)\mapsto H_{\rho}(x,y,s,t)$ acts as the identity on $S\times\{0,1\}$ for $t\in[0,1]$.

For $s\in[0,1]$, $\widetilde{H}_{\rho}(\widetilde{\rho}(s),1)$ lies on $\{(x,0,t):x,t\in[0,1]\}=\widetilde{\sigma}\times[0,1]\subseteq\widetilde{S}\times[0,1]$.
This means that, by the end of the isotopy $H_{\rho}$, the arc $\rho\subseteq T_{\sigma}$ has been carried to a simple arc within $\sigma\times[0,1]\subseteq S\times[0,1]$ connecting the two components of $\partial (\sigma\times[0,1])$.
In addition, $\partial T_{\sigma}=\partial (\sigma\times[0,1])$.
By \cite{MR0248844} Chapter IV Lemma 4.6, there is an ambient isotopy fixed on $\partial(S\times[0,1])\cup \rho$ putting $T_{\sigma}$ into general position with respect to $\sigma\times[0,1]$.
Since $S\times[0,1]$ is irreducible and $T_{\sigma}$ is incompressible in $S\times[0,1]$, there is an ambient isotopy, fixed on $\partial(S\times[0,1])\cup \rho$, that makes $T_{\sigma}$ coincide with $\sigma\times[0,1]$.

By Lemma \ref{isotopeannuluslemma}, there is an isotopy of $T_{\sigma}$, fixed on $\sigma\times\{1\}$, after which $\sigma\times[0,1]$ is mapped by $\psi_{\partial}$.
Using Lemma \ref{bdyisotopylemma}, the restriction of this isotopy to $\sigma\times\{0\}$ can be extended to one on $S\times\{0\}$ that is fixed on $(R\cap S)\setminus\sigma$, and then the resulting isotopy of $T_{\sigma}\cup (S\times\{0,1\})$ (acting by the identity on $S\times\{1\}$) can be extended to an ambient isotopy in $S\times[0,1]$.

We now turn our attention to the second annulus $T_{\sigma'}$.
This intersects $T_{\sigma}$ along the arc $\psi_S(p')\times[0,1]=\psi_{\partial}(\{p'\}\times[0,1])$.
The curve $T_{\sigma'}\cap (S\times\{1\})$ is $\psi_S(\sigma')\times\{1\}=\psi_{\partial}(\sigma'\times\{1\})$.
The other boundary curve $T_{\sigma'}\cap(S\times\{1\})$ is therefore homotopic to $\psi_S(\sigma')\times\{0\}$ by a homotopy keeping $(\psi_S(p'),0)$ fixed.
Moreover, it meets each circle in $R\cap S$ exactly once, and these points break it into simple arcs, each of which connects the two boundary components of the annulus of $S\setminus R$ in which it lies.

The next step is to see that there is an isotopy of $S$, keeping $\sigma$ fixed and keeping $R\cap S$ invariant, that ends at $\psi_S$.
Using Lemma \ref{isotopeannuluslemma}, combined with Lemma \ref{bdyisotopylemma}, on all but one of these annuli in order around $\psi_S(\sigma')\times\{0\}$, we can build an isotopy of $S$ that keeps $\sigma$ fixed, keeps $R\cap S$ invariant and ends at the map $\psi_S$ on all but one annulus of $S\setminus R$.
The arc across this final annulus is homotopic, relative to its endpoints, to the remaining arc of $\psi_S(\sigma')$. Applying Theorem \ref{homotopicarcsthm} and Theorem \ref{alextrickthm} shows that a suitable isotopy of $S$ exists.
After extending this by the identity to $T_{\sigma}\cup (S\times\{0,1\})$, we may again apply Lemma \ref{bdyisotopylemma} to extend this to an ambient isotopy of $S\times[0,1]$.

We have now correctly positioned $\partial T_{\sigma'}$ as well as the arc $\rho=T_{\sigma}\cap T_{\sigma'}$.
By \cite{MR0248844} Chapter IV Lemma 4.6, we may isotope $T_{\sigma'}$, keeping $T_{\sigma}\cup(S\times\{0,1\})$ fixed, to put it into general position with respect to $\psi_S(\sigma')\times[0,1]=\psi_{\partial}(\sigma'\times[0,1])$.
Since $S\times[0,1]\setminus T_{\sigma}$ is irreducible and $T_{\sigma'}\setminus\rho$ is incompressible in $S\times[0,1]\setminus T_{\sigma}$, there is an ambient isotopy, fixed on $T_{\sigma}\cup(S\times\{0,1\})$, that makes $T_{\sigma'}$ coincide with $\psi_{\partial}(\sigma'\times[0,1])$.
Because $T_{\sigma'}\setminus\rho$ is a disc, Theorem \ref{alextrickthm} again gives that there is an isotopy of $T_{\sigma'}$, keeping $\partial T_{\sigma'}\cup\rho$ fixed, that ends at $\psi_{\partial}$. After extending by the identity to $T_{\sigma}\cup(S\times\{0,1\})$, this isotopy can be extended to an ambient isotopy of $S\times[0,1]$, again by Lemma \ref{bdyisotopylemma}.
Applying Theorem \ref{alextrickthm} to the $3$--ball $(S\times[0,1])\setminus(T_{\sigma}\cup T_{\sigma'})$ gives a final isotopy to $\psi_{\partial}$.

Hence we have constructed a suitable isotopy $H_2$.
\end{proof}

\begin{remark}
Since the isotopies $H_2$ and $H_3$ in the proof of Lemma \ref{fixboundarylemma} are supported on $\nhd(S)$, the result can in fact be applied simultaneously to any combination of the toral boundary components of $M$.
\end{remark}

\begin{corollary}
Let $M$, $S$, $R$, $H$ and $\psi$ be as in Lemma \ref{fixboundarylemma}.
Then there is an isotopy $H''\colon (M\times[0,1],\partial M)\to (M,\partial M)$ from the identity on $M$ to a homeomorphism $\psi'\colon (M,\partial M)\to (M,\partial M)$ such that $\psi'(R)=\psi(R)$ and $x\mapsto H''(x,t)$ restricts to the identity on $S$ for all $t\in[0,1]$.
\end{corollary}
\begin{proof}
By Lemma \ref{fixboundarylemma}, we may assume that $H(R\cap S,t)=R\cap S$ for $t\in[0,1]$.
Choose a product neighbourhood $\nhd(S)=S\times[0,1]$ of $S$ with $S=S\times\{0\}$ and such that $\psi(R)\cap\nhd(S)=(R\cap S)\times[0,1]$. 
For $t\in[0,1]$, let $h_t\colon S\to S$ be the homeomorphism $y\mapsto \pi_S(H((y,0),t))$, where $\pi_S\colon S\times[0,1]\to S$ is projection onto the first factor.

Define $H_{\partial}\colon M\times[0,1]\to M$ by, for $x\in M$ and $t\in[0,1]$,
\[
H_{\partial}(x,t)=
\begin{cases}
x & x\in M\setminus\nhd(S),\\
(h_{t-s}^{-1}(y),s)&x=(y,s)\in S\times[0,1], s\leq t,\\
(h_0^{-1}(y),s) & x=(y,s)\in S\times[0,1],t<s.
\end{cases}
\]
Since $h_0$ is the identity map on $S$, this is continuous.

Now define $H''$ by $H''(x,t)=H_{\partial}(H(x,t),t)$ for $x\in M$ and $t\in[0,1]$.
For $y\in S$ and $t\in[0,1]$,
\[
H''((y,0),t)=H_{\partial}(H((y,0),t),t)=H_{\partial}((h_t(y),0),t)=(h_t^{-1}(h_t(y)),0)=(y,0),
\]
since $H(S,t)=S$.
Thus $H''$ keeps $S$ fixed throughout.

Let $x\in\psi(R)$.
If $x\in M\setminus\nhd(S)$ then there exists $x'\in R$ with $H(x',1)=x$.
Then $\psi'(x')=H''(x',1)=H_{\partial}(x,1)=x$, and so $x\in\psi'(R)$.
On the other hand, if $x=(y,s)\in S\times[0,1]$ then
$x\in(R\cap S)\times[0,1]$, so $y\in R\cap S$.
This implies that $h_{1-s}(y)\in R\cap S$, and so $(h_{1-s}(y),s)\in\psi(R)$.
Therefore there exists $x''\in R$ with $H(x'',1)=(h_{1-s}(y),s)$.
Then $\psi'(x'')=H''(x'',1)=H_{\partial}(h_{1-s}(y),s)=(h_{1-s}^{-1}(h_{1-s}(y)),s)=(y,s)=x$.
Again, therefore, $x\in\psi'(R)$.

Conversely, let $x\in\psi'(R)$. Then there exists $x'\in R$ with $H_{\partial}(H(x',1),1)=x$.
If $H(x',1)\in M\setminus\nhd(S)$ then $x=H_{\partial}(H(x',1),1)=H(x',1)=\psi(x')$, showing that $x\in\psi(R)$.
If instead $\psi'(x')=H(x',1)=(y',s')\in S\times[0,1]$ then $(y',s')\in \psi(R)=(R\cap S)\times[0,1]$, so $y'\in R\cap S$.
Thus $h_{1-s'}(y')\in R\cap S$, and $(h_{1-s'}(y'),s')\in \psi(R)$.
Accordingly, there exists $x''\in R$ with $\psi(x'')=(h_{1-s'}(y'),s')$.
Hence $x=H_{\partial}(H(x',1),1)=H_{\partial}((y',s'),1)=(h_{1-s'}(y'),s')=\psi(x'')$, and so $x\in\psi(R)$.

Therefore $\psi'(R)=\psi(R)$.
\end{proof}

\section{Exact sequences}\label{sequencessection}
%master document is splitlinks.tex

Fix a point $p\in\Int(M)$.

\begin{definition}
For a manifold $M'$, we define the following notation.
Let $\homeo(M')$ denote the group of homeomorphisms from $M'$ to $M'$, and let $\homeo _0(M')$ be the subgroup of maps isotopic to the identity on $M'$.
Denote by $\homeo_{\partial}(M')$ the subgroup consisting of those maps that are the identity on $\partial M'$, and by $\homeo_{0,\partial}(M')$ those maps that are isotopic to the identity by an isotopy that is fixed on $\partial M$.
Set $\Mod(M')=\homeo(M')/\homeo_0(M')$ and $\Mod_{\partial}(M')=\homeo_{\partial}(M')/\homeo_{0,\partial}(M')$. That is, $\Mod(M')$ is the group of isotopy classes of homeomorphisms from $M'$ to itself, whereas $\Mod_{\partial}(M')$ is made up of homeomorphisms fixing the boundary up to isotopy fixing the boundary.

If $M'$ is orientable, define $\homeo^+(M')$, $\homeo^+_0(M')$, $\homeo^+_{\partial}(M')$ and $\homeo^+_{0,\partial}(M')$ to be the analogous groups where all maps are required to be orientation-preserving, as well as $\Mod^+(M')$ and $\Mod^+_{\partial}(M')$.
\end{definition}

\begin{remark}
Other similar groups could be defined, for example by considering maps that keep only some of the boundary components of $M'$ fixed. Several of the results below would also hold in such cases, but we will not discuss them explicitly.
\end{remark}

\begin{proposition}\label{sequenceprop}
There is an exact sequence of groups
\[
\pi_1(M,p)\to \Mod(M_p)\to\Mod(M)\to 1.
\]
The same is true if $\Mod$ is replaced by $\Mod_{\partial}$, or by either $\Mod^+$ or $\Mod^+_{\partial}$ if $M$ is orientable.
\end{proposition}
\begin{proof}
The first map $\Phi\colon\pi_1(M,p)\to\Mod(M_p)$ is given by $\Phi(g)= [\phi_{\rho}]$, where $\rho$ is any simple closed curve in $\Int(M)$ based at $p$ with $[\rho]=g$.
By Lemma \ref{welldefinedlemma}, $[\phi_{\rho}]$ is a well-defined element of $\Mod(M_p)$ given $\rho$. 
Corollary \ref{pwembeddedcor} shows that the choice of simple closed curve $\rho$ representing $g$ does not affect $[\phi_{\rho}]$. Thus the map $\Phi$ is well-defined. 

The second map $\Psi\colon\Mod(M_p)\to\Mod(M)$ is given by `filling in' the point $p$, taking $[\phi_{\rho}]$ viewed as maps of $M_p$ to $[\phi_{\rho}]$ viewed as maps of $M$. Note that this is not necessarily an isomorphism. The isotopies used to define $\Mod(M)$ may move $p$, whereas those used to define $\Mod(M_p)$ may not.

Corollary \ref{fixpointcor} shows that $\Psi$ is surjective, and hence the sequence is exact at $\Mod(M)$.
By Corollary \ref{inkernelcor}, $\im(\Phi)\subseteq\ker(\Psi)$. 
On the other hand, Lemma \ref{inimagelemma} gives that $\ker(\Psi)\subseteq\im(\Phi)$.
Thus the sequence is also exact at $\Mod(M_p)$.
\end{proof}

\begin{corollary}
Suppose that $M$ is orientable and $K\leq\pi_1(M,p)$.
There is an exact sequence
\[1\to K\to \pi_1(M,p)\to\Mod(M_p)\to\Mod(M)\to 1\]
if and only if there is an exact sequence
\[1\to K\to \pi_1(M,p)\to\Mod^+(M_p)\to\Mod^+(M)\to 1.\]
Similarly, there is an exact sequence
\[1\to K\to \pi_1(M,p)\to\Mod_{\partial}(M_p)\to\Mod_{\partial}(M)\to 1\]
if and only if there is an exact sequence
\[1\to K\to \pi_1(M,p)\to\Mod_{\partial}^+(M_p)\to\Mod_{\partial}^+(M)\to 1.\]
\end{corollary}
\begin{proof}
For each of these four sequences, the sequence is exact if and only if $K\cong\ker(\Phi)$, where $\Phi$ is as given in the proof of Proposition \ref{sequenceprop}.
The definition of $\Phi$ is dependent on which sequence is being considered. 
However, the definition of $\ker(\Phi)$ is the same when considering $\Mod$ as when considering $\Mod^+$. Similarly, the definition of $\ker(\Phi)$ is the same when considering $\Mod_{\partial}$ as when considering $\Mod_{\partial}^+$.
\end{proof}

\begin{remark}
We will denote by $\Phi_M$ the map $\Phi$ given in the proof of Proposition \ref{sequenceprop}, unless the  choice of manifold $M$ is clear, in which case we will denote the map by $\Phi$. 

Our primary concern for most of the rest of this paper is to understand $\ker(\Phi_M)$ for different manifolds $M$, particularly $3$--manifolds. 
If $\ker(\Phi_M)\cong 1$ then the Birman exact sequence holds for $M$. 
\end{remark}

\begin{lemma}\label{fixbdypointlemma}
Assume $\partial M\neq\emptyset$, and fix $p'\in\partial M$. 
Let $\rho$ be a simple closed curve in $\Int(M)$ based at $p$. Suppose that $\phi_{\rho}$ is isotopic to the identity on $M$ by an isotopy that keeps $p$ fixed and takes $p'$ around a loop that is null-homotopic in $\partial M$. Then $\rho$ is null-homotopic in $M$. In particular, $\rho$ is null-homotopic if the isotopy fixes $\partial M$.
\end{lemma}
\begin{proof}
Let $H\colon(M\times[0,1],\partial M\times[0,1],p\times[0,1])\to(M,\partial M,p)$ be the isotopy from $\phi_{\rho}$ to the identity, and define $\rho'\colon[0,1]\to\partial M$ by $\rho'(t)=H(p',t)$.
By hypothesis, $\rho'$ is null-homotopic in $\partial M$.

Let $\sigma$ be an embedded path from $p'$ to $p$.
Then $\phi_{\rho}(\sigma)$ is an embedded path from $p'$ to $p$ that is homotopic relative to its endpoints to $\sigma\cdot\rho$.
Now $H$ shows that, relative to its endpoints, the path $(\rho')^{-1}\cdot\phi_{\rho}(\sigma)$ is homotopic to $\sigma$.
Therefore,
\[
\rho\simeq \sigma^{-1}\cdot\sigma\cdot\rho\simeq\sigma^{-1}\cdot\phi_{\rho}(\sigma)\simeq\sigma^{-1}\cdot\rho'\cdot(\rho')^{-1}\cdot\phi_{\rho}(\sigma)\simeq\sigma^{-1}\cdot\rho'\cdot\sigma\simeq\sigma^{-1}\cdot\sigma.
\]
Therefore $\rho$ is null-homotopic.
\end{proof}

\begin{corollary}
If $\partial M\neq\emptyset$ then there is an exact sequence
\[
1\to\pi_1(M,p)\to \Mod_{\partial}(M_p)\to\Mod_{\partial}(M)\to 1.
\]
\end{corollary}
\begin{proof}
Given Proposition \ref{sequenceprop}, it remains only to check that $\Phi$ is injective.
Lemma \ref{fixbdypointlemma} shows that $\ker(\Phi)\cong 1$.
\end{proof}

\begin{remark}
In Definition \ref{homeodefn}, we required that the isotopy used to define $\phi_{\rho}$ keeps $\partial M$ fixed. From now only we will only be concerned with studying $\Mod$. In this case, Lemma \ref{replaceboundarylemma} shows that we may drop this condition from Definition \ref{homeodefn}.
\end{remark}

\section{Surfaces and fibred $3$--manifolds}\label{fibredsection}
%master document is splitlinks.tex

Many of the ideas used up to this point also apply for surfaces, except that in this case we cannot require the closed curves we consider to be simple. 
We must also remove the dependence on Theorem \ref{homotoisothm}, which is used to show that if $\rho$ and $\rho'$ are homotopic curves then $\phi_{\rho}$ and $\phi_{\rho'}$ lie in the same isotopy class, or equivalently that $\phi_{\rho'}^{-1}\circ\phi_{\rho}$ is isotopic to the identity fixing $p$.
We prove this for surfaces using the following results.

\begin{lemma}
Suppose $M$ is a surface.
Let $\rho,\rho'\colon[0,1]\to\Int(M)$ be 
paths with $\rho(0)=\rho'(0)=\rho(1)=\rho'(1)=p$.
Suppose there exists a homotopy $H\colon[0,1]\times[0,1]\to\Int(M)$ from $\rho$ to $\rho'$ with $H(0,t)=H(1,t)=p$  for $t\in[0,1]$.
Then $\phi_{\rho}$ is isotopic to $\phi_{\rho'}$ by a isotopy that keeps $p$ fixed.
\end{lemma}
\begin{proof}
Let $H_{\rho}\colon(M\times[0,1],\partial M\times[0,1])\to(M,\partial M)$ be the isotopy used to define $\phi_{\rho}$.
Using $H_{\rho}$ and $H$, we will build a homotopy $H'\colon(M\times[0,1],\partial M\times[0,1])\to(M,\partial M)$ with $H'(x,0)=x$ for $x\in M$ and $H'(p,t)=\rho'(t)$ for $t\in[0,1]$.
%Intuituvely, the homotopy $H'$ is built as follows.
%At time $t$, take two concentric circles around $p$ in $M$.
%Do $H_{\rho}$ outside the outer one. Collapse the inner one to use the annulus between them to do the rest of $H_{\rho}$. Use the disc inside the inner circle to run along the path given by $H$ at time $t$.
%The problem is doing this in a PL fashion.

Let $X$ be the disc in $\mathbb{R}^2$ whose boundary is the triangle with corners at $(0,0)$, $(0,1)$ and $(1,\frac{1}{2})$.
Define $g\colon[0,1]\times[0,1]\to X$ by
\[g(s,t)=
\begin{cases}
(0,t)& t\in\{0,1\},\\
(2st,t)&t\leq\frac{1}{2},\\
(2s(1-t),t)& t>\frac{1}{2}.
\end{cases}\]
Choose a triangulation of $[0,1]\times[0,1]$ with respect to which $H$ is simplicial. 
Let $g'\colon[0,1]\times[0,1]\to X$ be the simplicial map where each vertex is mapped by $g$.
This gives a triangulation of $X$.
Define $H_X\colon X\to \Int(M)$ to be the simplicial map that maps each vertex by $H\circ (g')^{-1}$.
Note that $g$ is injective on the vertices except for those in $[0,1]\times\{0,1\}$.
On the other hand, $H$ maps every point of $[0,1]\times\{0,1\}$ to $p$. Therefore $H_X$ is well-defined.

Let $Y\subseteq X$ be the two edges of the triangle bounding $X$ that contain the point $(1,\frac{1}{2})$.
Then the restriction of $g'$ to $\{1\}\times[0,1]$ is a level-preserving bijection to $Y$.
Thus, if $(s,t)\in Y$ then $H_X(s,t)=H(1,t)=\rho'(t)$.

Define an equivalence relation $\sim_Z$ on $[-2,2]\times[-2,2]\times[0,1]\sqcup X$ by
$(0,0,t)\sim_Z(0,t)$ for $t\in[0,1]$.
Let $Z=([-2,2]\times[-2,2]\times[0,1]\sqcup X)/\!\sim_Z$.
Choose an embedding $f\colon[-2,2]\times[-2,2]\to\Int(M)$ with $f(0,0)=p$.
Define $H_Z\colon Z\to M$ by, for $z\in Z$,
\[
H_Z(z)=
\begin{cases}
H_{\rho}(f(a,b),t) & z=(a,b,t)\in [-2,2]\times[-2,2]\times[0,1],\\
H_X(z) & z\in X.
\end{cases}
\]
Note that if $t\in[0,1]$ then $H_{\rho}(f(0,0),t)=H_{\rho}(p,t)=\rho(t)=H(0,t)=H_X(0,t)$, so $H_Z$ is well-defined.

We wish to define a surjection $\theta\colon [-2,2]\times[-2,2]\times[0,1]\to Z$.
We first triangulate $[-2,2]\times[-2,2]\times[0,1]$.
Set $A=\{-2,2\}\times\{-2,2\}\times\{0,\frac{1}{2},1\}$, $B= \{0\}\times\{0\}\times\{0,1\}$, $C=\{0\}\times\{0\}\times\{\frac{1}{2}\}$ and $D=\{-1,1\}\times\{-1,1\}\times\{\frac{1}{2}\}$.
We take the vertices to be $A\cup B\cup C\cup D$.
Add edges connecting both $(0,0,0)$ and $(0,0,1)$ to each point of $\left(\{-2,2\}\times\{-2,2\}\times\{\frac{1}{2}\}\right)\cup\left(\{0\}\times\{0\}\times\{\frac{1}{2}\}\right)$.
Add further edges to give a triangulation of $[-2,2]\times[-2,2]\times[0,1]$ without adding any more vertices.

Now define $\theta$ as the simplicial map with, for a vertex $v$,
\[
\theta(v)=
\begin{cases}
v & v\in A\cup B,\\
(1,\frac{1}{2}) & v\in C,\\
(0,0,\frac{1}{2}) & v\in D.
\end{cases}
\]
Finally, define $H'$ by, for $x\in M$ and $t\in[0,1]$,
\[
H'(x,t)=
\begin{cases}
H_{\rho}(x,t) & x\notin \im(f),\\
H_Z(\theta(f^{-1}(x),t)) & x\in\im(f).
\end{cases}
\]
For $x\in M$ and $(a,b)\in\partial ([-2,2]\times[-2,2])$ such that $f(a,b)=x$, and $t\in[0,1]$, $\theta(a,b,t)=(a,b,t)$ and $H_Z(\theta(a,b,t))=H_Z(a,b,t)=H_{\rho}(f(a,b),t)=H_{\rho}(x,t)$. Therefore $H'$ is continuous.

Let $x\in M$. If $x\notin\im(f)$ then $H'(x,0)=H_{\rho}(x,0)=x$.
On the other hand, if $x\in\im(f)$ then
$H'(x,0)=H_Z(\theta(f^{-1}(x),0))=H_Z(f^{-1}(x),0)=H_{\rho}(f(f^{-1}(x)),0)=x$.
Hence $H'$ begins at the identity map.

Now let $x\in\partial M$. Then $x\notin\im(f)$, so $H'(x,t)=H_{\rho}(x,t)\in\partial M$ for $t\in[0,1]$.

Finally, let $t\in[0,1]$.
Then $H'(p,t)=H_Z(\theta(f^{-1}(p),t))=H_Z(\theta(0,0,t))$.
Because $\theta(0,0,t)\in Y\subseteq X\subseteq Z$, 
we see that $H_Z(\theta(0,0,t))=H_X(\theta(0,0,t))=\rho'(t)$.

Now $H'$ could be used to define $\phi_{\rho'}$ except that it is a homotopy rather than an isotopy. 
However, essentially the same proof as Lemma \ref{welldefinedlemma} shows that $\phi_{\rho'}$ is homotopic to the map $x\mapsto H'(x,1)$ by a homotopy keeping $p$ fixed.
Note that, for $x\in M$,
if $x\notin\im(f)$ then $H'(x,1)=H_{\rho}(x,1)$, while if $x\in\im(f)$ then
$H'(x,1)=H_Z(\theta(f^{-1}(x),1))=H_Z(f^{-1}(x),1)=H_{\rho}(f(f^{-1}(x)),1)=H_{\rho}(x,1)$.
Therefore $\phi_{\rho}$ and $\phi_{\rho'}$ are homotopic keeping $p$ fixed. 
By Theorem \ref{surfacehomotoisothm}, they are then isotopic keeping $p$ fixed.
\end{proof}

Birman and Chillingworth noted in \cite{MR0300288} that mapping class groups of non-orientable manifolds can be studied by looking at the same ideas in the orientation double cover.

\begin{lemma}\label{reverseorientationlemma}
Suppose $M$ is non-orientable, and let $\rho$ be an orientation-reversing loop in $\Int(M)$ based at $p$.
Then $\phi_{\rho}$ is not isotopic to the identity by an isotopy keeping $p$ fixed.
\end{lemma}
\begin{proof}
Choose a map from $[-1,1]^n$ to $\Int(M)$ that is a homeomorphism onto its image and sends the origin to $p$.
For $0\leq\delta\leq 1$, denote by $B_{\delta}$ the image under this map of the $n$--ball $[-\delta,\delta]^n$.

Suppose $H\colon (M\times[0,1],\partial M\times[0,1],p\times[0,1])\to(M,\partial M,p)$ is an isotopy from $\phi_{\rho}$ to the identity.
Then there exists $\delta>0$ such that $H(\phi_{\rho}(B_{\delta})\times[0,1])\subseteq B_1$.
Since $\rho$ is orientation reversing, the map $\phi_{\rho}$ reverses the orientation of $B_{\delta}$.
Therefore the isotopy $H$ reverses the orientation of $B_{\delta}$. This is impossible within the orientable manifold $B_1$. Hence $\phi_{\rho}$ cannot be isotopic to the identity in $M_p$.
\end{proof}

\begin{lemma}\label{coverlemma}
Let $(\widehat{M},\widehat{p})$ be an $m$--fold cover of $(M,p)$ 
with a covering map $\pi\colon(\widehat{M},\widehat{p})\to (M,p)$.
Then there is a group $L\leq\pi_*(\ker(\Phi_{\widehat{M}}))$ 
such that
$L\trianglelefteq\ker(\Phi_M)$ and $|\ker(\Phi_M)/L|\leq m$.

If $\widehat{M}$ is a regular cover with covering group $J$ then 
$L\trianglelefteq\pi_*(\ker(\Phi_{\widehat{M}}))$ and $\ker(\Phi_M)/L$ is isomorphic to a normal subgroup of $J$. 

Moreover, if $M$ is non-orientable and $\widehat{M}$ is the orientation double cover of $M$ then $\ker(\Phi_M)\trianglelefteq\pi_*(\ker(\Phi_{\widehat{M}}))$.
\end{lemma}
\begin{proof}
Set $G=\pi_*(\pi_1(\widehat{M},\widehat{p}))$. Then $G\leq\pi_1(M,p)$ and $|\pi_1(M,p):G|=m$.

Set $L=G\cap\ker(\Phi_M)$.
Let $\rho$ be a loop in $\Int(M)$ based at $p$ such that $\rho\in L$.
Then there is an isotopy $H_{\rho}\colon(M\times[0,1],\partial M\times[0,1])\to(M,\partial M)$ from the identity on $M$ to itself that can be used to define $\Phi_M(\rho)$.
The loop $\rho$ lifts to a loop $\widehat{\rho}$ in $\Int(\widehat{M})$ based at $\widehat{p}$.
Also, the isotopy $H_{\rho}$ lifts to an isotopy $H_{\widehat{\rho}}\colon(\widehat{M}\times[0,1],\partial\widehat{M}\times[0,1])\to(\widehat{M},\partial\widehat{M})$ beginning at the identity on $\widehat{M}$ that can be used to define $\Phi_{\widehat{M}}(\widehat{\rho})$.
Note that $\Phi_{\widehat{M}}(\widehat{\rho})$ therefore differs from the identity on $\widehat{M}$ by a covering transformation. However, as $\Phi_{\widehat{M}}(\widehat{\rho})$ fixes $\widehat{p}$ this means that it is the identity map on $\widehat{M}$. Hence $\widehat{\rho}\in\ker(\Phi_{\widehat{M}})$.
Therefore $L\leq \pi_*(\ker(\Phi_{\widehat{M}}))$.

Now let $\rho$ and 
$\sigma$ be loops in $\Int(M)$ based at $p$ such that $\rho\in L$ and $\sigma\in\ker(\Phi_M)$.
Then $\sigma\cdot\rho\cdot\sigma^{-1}\in\ker(\Phi_M)$.

Since $\sigma\in\ker(\Phi_M)$,  there is an isotopy $H_{\sigma}\colon (M\times[0,1],\partial M\times[0,1])\to(M,\partial M)$ from the identity on $M$ to itself that can be used to define $\Phi_M(\sigma)$.
This lifts to an isotopy $\widehat{H}_{\sigma}\colon(\widehat{M}\times[0,1],\partial\widehat{M}\times[0,1])\to(M,\partial M)$ from the identity on $\widehat{M}$ to a lift of the identity on $M$.
Define $\psi\colon(\widehat{M},\partial \widehat{M})\to(\widehat{M},\partial\widehat{M})$
by $\psi(x)=\widehat{H}_{\sigma}(x,1)$.
Then $\psi$ is a covering transformation of $\widehat{M}$ corresponding to $\sigma$.

Let $\widehat{\rho}$ and 
$\widehat{\sigma}$ be the lifts of $\rho$ and 
$\sigma$ respectively 
to $\widehat{M}$ beginning at $\widehat{p}$, and
let $\widehat{p}'=\widehat{\sigma}(1)=\psi(\widehat{p})$. Then the lift $\widehat{\rho}'$ of $\rho$ beginning at $\widehat{p}'$ is given by $\widehat{\rho}'(t)=\psi(\widehat{\rho}(t))$. Since $\widehat{\rho}$ is a closed loop in $\widehat{M}$, so is $\widehat{\rho}'$.
Therefore the lift of $\sigma\cdot\rho\cdot\sigma^{-1}$ beginning at $\widehat{p}$ is $\widehat{\sigma}\cdot\widehat{\rho}'\cdot\widehat{\sigma}^{-1}$. This is a closed loop in $\widehat{M}$, so $\sigma\cdot\rho\cdot\sigma^{-1}\in G$.
Hence $L\trianglelefteq\ker(\Phi_M)$.

Define a map of sets $\theta\colon \ker(\Phi_M)/L\to\pi_1(M,p)/G$ by $\theta(\rho L)=\rho G$ for $\rho\in\ker(\Phi_M)$.
Suppose $\rho_1,\rho_2\in\ker(\Phi_M)$ with $\theta(\rho_1)=\theta(\rho_2)$.
Then $\rho_1\cdot\rho_2^{-1}\in G\cap \ker(\Phi_M)$, 
and so $\rho_1 L=\rho_2 L$.
Thus $\theta$ is injective and $|\ker(\Phi_M):L|\leq|\pi_1(M,p):G|=m$.

Now suppose that $\widehat{M}$ is a regular cover of $M$. Then $G\trianglelefteq\pi_1(M,p)$ and $\pi_1(M,p)/ G\cong J$.
Since in addition $\ker(\Phi_M)\trianglelefteq\pi_1(M,p)$, it follows that $L\trianglelefteq\pi_1(M,p)$.
Therefore $L\trianglelefteq\pi_*(\ker(\Phi_{\widehat{M}}))$.

The map $\theta$ is now a group homomorphism.
Let $\rho$ and $\sigma$ be loops in $\Int(M)$ based at $p$ such that $\rho\in\ker(\Phi_M)$ and $\sigma\in\pi_1(M,p)$.
Then $(\sigma G)\cdot\theta(\rho L)\cdot(\sigma G)^{-1}=\sigma\cdot\rho\cdot\sigma^{-1} G$.
As $\ker(\Phi_M)\trianglelefteq\pi_1(M,p)$, we see that $\sigma\cdot\rho\cdot\sigma^{-1}\in\ker(\Phi_M)$ and $(\sigma G)\cdot\theta(\rho L)\cdot(\sigma G)^{-1}\in\im(\theta)$.
Hence $\im(\theta)\trianglelefteq\pi_1(M,p)/G\cong J$.

Finally suppose that $M$ is non-orientable and $\widehat{M}$ is the orientation double cover of $M$.
Then $\pi_1(M,p)/G=\{g_1,g_2\}$, where $g_1$ contains all orientation-preserving elements of $\pi_1(M,p)$ and $g_2$ contains all the orientation-reversing elements.
By Lemma \ref{reverseorientationlemma}, if $\rho$ is a loop in $\Int(M)$ based at $p$ such that $\rho\in\ker(\Phi_M)$ then $\rho$ is orientation-preserving.
Thus $g_2\notin\im(\theta)$, and $\ker(\Phi_M)/L\cong 1$.
Therefore $L=\ker(\Phi_M)$.
\end{proof}

\begin{proposition}\label{surfacesprop}
Let $M$ be a compact, connected surface with $\chi(M)\neq 0$ (that is, $M$ is not a torus, an annulus, a M\"obius band or a Klein bottle).
Then there is an exact sequence
\[
1\to\pi_1(M,p)\to\Mod(M_p)\to\Mod(M)\to 1.
\]
\end{proposition}
\begin{proof}
By Proposition \ref{sequenceprop}, we only need to check that $\ker(\Phi_M)\cong 1$.
If $M$ is non-orientable, its orientation double cover $\widehat{M}$ is a compact, connected, orientable surface with $\chi(\widehat{M})\neq 0$.
Lemma \ref{coverlemma} therefore allows us to reduce to the case that $M$ is orientable.

If $\chi(M)> 0$
then $M$ is either a sphere or a disc. Thus $\pi_1(M,p)\cong 1$, and injectivity of $\Phi_M$ is immediate.
On the other hand, if $\chi(M)<0$ then injectivity of $\Phi_M$ is given by Theorem \ref{birmanthm}.
\end{proof}

\begin{lemma}\label{changemonodromylemma}
Let $N$ be a manifold, and let $f,g\colon (N,\partial N)\to (N,\partial N)$ be homeomorphisms.
Set $M_f=(N\times[0,1])/\!\sim_f$ and $M_g=(N\times[0,1])/\!\sim_g$, where $(x,1)\sim_f (f(x),0)$ and $(x,1)\sim_g(g(x),0)$ for $x\in N$.
If $f$ is isotopic to $g$ then $M_f$ is homeomorphic to $M_g$.
\end{lemma}
\begin{proof}
Let $H\colon (N\times[0,1],\partial N\times[0,1])\to (N,\partial N)$ be an isotopy from $f$ to $g$.
Define $\psi\colon M_f\to M_g$ by
$\psi(x,t)=(H(f^{-1}(x),1-t),t)$ for $x\in N$ and $t\in[0,1]$.
Note that $\psi(f(x),0)=(H(x,1),0)=(g(x),0)\sim_g(x,1)=(H(f^{-1}(x),0),1)=\psi(x,1)$ for $x\in N$, so $\psi$ is a well-defined map from $M_f$ to $M_g$.
In addition, for $t\in[0,1)$, the map $x\mapsto \psi(x,t)$ is a homeomorphism of $N\times\{t\}$. Hence $\psi$ is a homeomorphism. 
\end{proof}

\begin{lemma}\label{pushroundfibrelemma}
Suppose $M$ is a fibred manifold with base space $N$ and monodromy $\omega\colon (N,\partial N)\to (N,\partial N)$.
Assume $p$ lies on $N$.
Express $M$ as $(N\times[0,1])/\!\sim_{\omega}$, where $(x,1)\sim_{\omega}(\omega(x),0)$ for $x\in N$.
Define $\rho_{\omega}\colon[0,1]\to M$ by $\rho_{\omega}(t)=(p,t)$.
Then $\phi_{\rho_{\omega}}$ is the map $(x,s)\mapsto(\omega(x),s)$.
\end{lemma}
\begin{remark}
Note that here we do not really need $\rho_{\omega}$ to be a closed loop, so it is not necessary that $\omega(p)=p$.
\end{remark}
\begin{proof}
Define $H\colon(M\times[0,1],\partial M\times[0,1])\to(M,\partial M)$ by, for $x\in N$, $s\in[0,1)$ and $t\in[0,1]$,
\[
H(x,s,t)=
\begin{cases}
(x,s+t) & s+t\leq 1 ,\\
(\omega(x),s+t-1) &1< s+t.
\end{cases}
\]
Note that $H(x,s,1-s)=(x,1)\sim_{\omega}(\omega(x),0)$
and
$H(\omega(x),0,t)=(\omega(x),t)$ for $x\in N$ and $t\in[0,1]$.
Therefore $H$ is well-defined and continuous.
Moreover, $H(p,0,t)=(p,t)=\rho_{\omega}(t)$ for $t\in[0,1]$.
Therefore $H$ can be used to define $\phi_{\rho_{\omega}}$.
In addition, 
$H(x,s,1)=(\omega(x),s)$.
\end{proof}

\begin{lemma}\label{mustbeperiodiclemma}
Suppose that $M$ is fibred with base space 
$N$ and with monodromy $\omega\colon (N,\partial N)\to (N,\partial N)$.
Express $M$ as $(N\times[0,1])/\!\sim_{\omega}$, where $(x,1)\sim_{\omega}(\omega(x),0)$ for $x\in N$.
Assume $p$ lies on $N$ and $\omega(p)=p$.
Define $\rho_{\omega}\colon[0,1]\to M$ by $\rho_{\omega}(t)=(p,t)$.
Let $\rho_N\colon[0,1]\to M$ be a loop in $\Int(N)$ based at $p$. 

If $\rho_N\in\ker(\Phi_M)$ then $\omega(\rho_N)$ is homotopic to $\rho_N$ in $N$, keeping $p$ fixed.

If $N$ is a surface (so $M$ is a $3$--manifold) and $\rho_{\omega}^m\cdot\rho_N\in\ker(\Phi_M)$ for some $m\in\mathbb{Z}$ then 
$\phi_{\rho_N}\circ\omega^m$ is isotopic to the identity on $N$ keeping $p$ fixed.
In particular, $\omega^m$ is isotopic to the identity on $N$.
\end{lemma}
\begin{proof}
Let $\widetilde{M}=N\times\mathbb{R}$ and $\widetilde{p}=(p,0)$.
Define $\pi\colon(\widetilde{M},\partial\widetilde{M},\widetilde{p})\to (M,\partial M,p)$ by requiring that
$\pi(x,s)=(x,s)$ and $\pi(x,r+s)=\pi(\omega(x),r-1+s)$ for $x\in N$, $s\in[0,1)$ and $r\in\mathbb{Z}$.
Then $(\widetilde{M},\widetilde{p})$ is the infinite cyclic cover of $(M,p)$.
Let $\pi_N\colon\widetilde{M}\to N$ be 
projection onto the first factor.

First suppose that $\rho_N\in\ker(\Phi_M)$.
Let $\widetilde{\rho}_N\colon[0,1]\to\widetilde{M}$ be the lift of $\widetilde{\rho}$ beginning at $\widetilde{p}$, and let $\widetilde{\rho}'_N\colon[0,1]\to\widetilde{M}$ be the lift beginning at $(p,1)$.
For $t\in[0,1]$, $\widetilde{\rho}_N(t)=(\rho_N(t),0)$ and $\pi(\widetilde{\rho}'_N(t))=\rho_N(t)=\pi(\widetilde{\rho}_N(t))$, 
so $\widetilde{\rho}'_N(t)=(\omega^{-1}(\rho_N(t)),1)$.
In particular, $\pi_N(\widetilde{\rho}'_N(t))=\omega^{-1}(\rho_N(t))$ for $t\in[0,1]$.

As $\rho_N\in\ker(\Phi_M)$, there exists an isotopy $H_{\rho_N}\colon(M\times[0,1],\partial M\times[0,1])\to (M,\partial M)$ from the identity on $M$ to itself such that $H_{\rho_N}(p,0,t)=(\rho_N(t),0)$ for $t\in[0,1]$.
This lifts to an isotopy
$\widetilde{H}_{\rho_N}\colon(\widetilde{M}\times[0,1],\partial\widetilde{M}\times[0,1])\to(\widetilde{M},\partial\widetilde{M})$ from the identity on $\widetilde{M}$ to itself
such that $\widetilde{H}_{\rho_N}(p,0,t)=\widetilde{\rho}_{N}(t)$
and $\widetilde{H}_{\rho_N}(p,1,t)=\widetilde{\rho}'_N(t)$
for $t\in[0,1]$.

Define $H'_{\rho_N}\colon[0,1]\times[0,1]\to N$ by
$H'_{\rho_N}(s,t)=\pi_N(\widetilde{H}_{\rho_N}(p,t,s))$ for $s,t\in[0,1]$.
Note that, for $t\in[0,1]$, 
\[H'_{\rho_N}(0,t)=\pi_N(\widetilde{H}_{\rho_N}(p,t,0))=\pi_N(p,t)=p=\pi_N(\widetilde{H}_{\rho_N}(p,t,1))=H'_{\rho_N}(1,t).\] Therefore, $s\mapsto H'_{\rho_N}(s,t)$ is a loop in $N$ based at $p$ for $t\in[0,1]$.
Hence $H'_{\rho_N}$ is a homotopy, keeping $p$ fixed, from $\rho_N$ to the path $\rho'_N\colon[0,1]\to N$ given by, for $s\in[0,1]$,
\[
\rho'_N(s)=\pi_N(\widetilde{H}_{\rho_N}(p,1,s))=\pi_N(\widetilde{\rho}'_N(s))=\omega^{-1}(\rho_N(s)).
\]

Now suppose instead that $N$ is a surface and $\rho_{\omega}^m\cdot\rho_N\in\ker(\Phi_M)$ for some $m\in\mathbb{Z}$.
Let $H_N\colon(N\times[0,1],\partial N\times[0,1])\to(N,\partial N)$ be the isotopy used to define $\Phi_N(\rho_N)$ on $N$. We may assume that $H_N$ keeps $\partial N$ fixed throughout.
Take the isotopy $H_N$ on $N\times\{0\}$ together with the identity on $\partial M$, and extend this, using Lemma \ref{bdyisotopylemma}, to an isotopy $H\colon(M\times[0,1],\partial M\times[0,1])\to(M,\partial M)$. Then $H$ can be used to define $\Phi_M(\rho_N)$,
and the restriction of $\Phi_M(\rho_N)$ to $N=N\times\{0\}$ is $\Phi_N(\rho_N)$.

Thus, by Lemma \ref{pushroundfibrelemma}, the restriction of $\Phi_M(\rho)=\Phi_M(\rho_{\omega}^m\cdot\rho_N)$ to $N$ is $\phi_{\rho_N}\circ\omega^m$.
Let $H_{\rho}\colon (M\times[0,1],\partial M\times[0,1],p\times[0,1])\to(M,\partial M,p)$ be the isotopy from the identity on $M$ to $\Phi_M(\rho_{\omega}^m\cdot\rho_N)$.
This lifts to an isotopy $\widetilde{H}_{\rho}\colon(\widetilde{M}\times[0,1],\partial\widetilde{M}\times[0,1],\widetilde{p}\times[0,1])\to(\widetilde{M},\partial\widetilde{M},\widetilde{p})$.

Define $\widehat{H}_N\colon (N\times[0,1],\partial N\times[0,1],p\times[0,1])\to(N,\partial N,p)$ by $\widehat{H}_N(x,t)=\pi_N(\widetilde{H}_{\rho}(x,t))$.
Then $\widehat{H}_N$ is a homotopy, keeping $p$ fixed, from the identity on $N$ to $\phi_{\rho}\circ\omega^m$.
By Theorem \ref{surfacehomotoisothm}, there is therefore an isotopy, keeping $p$ fixed, from the identity on $N$ to $\phi_{\rho_N}\circ\omega^m$.
As $\phi_{\rho_N}$ is isotopic to the identity on $N$, it follows that $\omega^m$ is isotopic to the identity on $N$.
\end{proof}

\begin{remark}
Suppose that the monodromy $\omega\colon (N,\partial N)\to (N,\partial N)$ is such that  $\omega^m$ is isotopic to the identity on $N$ for some $m\neq 0$.
If in fact $\omega^m$ is the identity on $M$ then Lemma \ref{pushroundfibrelemma} shows that $\rho_{\omega}^m\in\ker(\Phi_M)$. It therefore seems natural to ask whether the monodromy may always be chosen in such a way. That is, if $\omega^m$ is isotopic to the identity for some $m>0$, does there exist a homeomorphism $\omega'\colon N\to N$ isotopic to $\omega$ on $N$ such that $(\omega')^m$ is the identity on $N$?
The Nielsen Realisation Theorem says that this is the case in the topological and smooth categories (see, for example, \cite{MR2850125} Theorem 7.2, \cite{MR690845} Theorem 5 and \cite{MR1189862} Corollary 8.1), but the author does not know of a complete proof for the piecewise linear category.
\end{remark}
%We will see below though that if such a space is orientable then it is Seifert fibred. Since Seifert fibrations are nearly always unique (see Jaco VI.16), it seems not unreasonable that the answer should also also `yes' in the PL category.

\begin{corollary}\label{notperiodiccor}
Suppose $M$ is a $3$--manifold and is fibred with base space a surface $N$ 
and with monodromy $\omega\colon (N,\partial N)\to (N,\partial N)$.
Suppose that $\omega^m$ is not isotopic to the identity on $N$ for any $m\neq 0$.
If $N$ is a torus then $\ker(\Phi_M)\trianglelefteq\pi_1(N)\cong\mathbb{Z}^2$, and otherwise
there is an exact sequence
\[
1\to\pi_1(M,p)\to\Mod(M_p)\to\Mod(M)\to 1.
\]
\end{corollary}
\begin{proof}
We first show that if $N$ in not a torus then $\chi(N)\neq 0$.
If $\chi(N)=0$ then $N$ is a torus, an annulus, a M\"obius band or a Klein bottle.
Suppose $N$ is an annulus. Then, as $\omega^2$ is orientation-preserving,  \cite{MR0214087} Theorem 5.6 says that $\omega^2$ is isotopic to the identity on $N$, contrary to hypothesis.
If $N$ is a Klein bottle, \cite{MR0145498} Lemma 5 gives that $\Mod(N)\cong\mathbb{Z}_2\times\mathbb{Z}_2$, so $\omega^2$ must be isotopic to the identity.
Suppose finally that $N$ is a M\"obius band. The restriction of $\omega^2$ to $\partial N$ is orientation-preserving and so is isotopic to the identity map on $\partial N$, by Theorem \ref{spheremapthm}. 
By Lemma \ref{bdyisotopylemma}, this isotopy can be extended to an isotopy of $N$. Thus we may assume that the restriction of $\omega^2$ to $\partial N$ is the identity.
Therefore \cite{MR0214087} Theorem 3.4 gives that $\omega^2$ is isotopic to the identity on $N$. This again contradicts the hypotheses.
Hence we see that $\chi(N)\neq 0$ unless $N$ is a torus.

Express $M$ as $(N\times[0,1])/\!\sim_{\omega}$, where $(x,1)\sim_{\omega}(\omega(x),0)$ for $x\in N$.
We may assume that $p$ lies on $N$ and that $\omega(p)=p$.
Let $\rho_{\omega}\colon[0,1]\to \Int(M)$ be the path $\rho_{\omega}(t)=(p,t)$.
Then every element of $\pi_1(M,p)$ can be expressed as $\rho_{\omega}^m\cdot\rho_N$ for some $m\in\mathbb{Z}$ and some loop $\rho_N$ based at $p$ and lying in $\Int(N)$.

Choose a loop $\rho_N$ in $\Int(N)$ based at $p$ and $m\in\mathbb{Z}$ such that $\rho_{\omega}^m\cdot\rho_N\in\ker(\Phi_M)$.
Lemma \ref{mustbeperiodiclemma} implies that $\omega^m$ is isotopic to the identity on $N$.
By hypothesis this means that $m=0$. 
Hence $\ker(\Phi_M)\leq\pi_1(N,p)$. As $\ker(\Phi_M)\trianglelefteq\pi_1(M,p)$, it follows automatically that $\ker(\Phi_M)\trianglelefteq\pi_1(N,p)$.

Now assume $N$ is not a torus.
Again from Lemma \ref{mustbeperiodiclemma}, we see that $\Phi_{N}(\rho_N)$ is isotopic to the identity on $N$ keeping $p$ fixed. That is, $\rho_N\in\ker(\Phi_N)$.
As $\chi(N)\neq 0$, Proposition \ref{surfacesprop} shows that $\rho_N$ is null-homotopic in $N$.
Thus $\rho$ is null-homotopic in $M$.
Therefore $\ker(\Phi_M)\cong 1$.
\end{proof}

\begin{proposition}\label{productbundleprop}
Suppose $M$ is fibred with base space $N$ and monodromy $\omega\colon (N,\partial N)\to (N,\partial N)$.
Suppose that $\omega$ is isotopic to the identity on $N$.
Then there is an exact sequence
\[\pi_1(N,p)\to\Mod(M_p)\to\Mod(M)\to 1.\]

If $N$ is a compact surface then there is an exact sequence
\[
1\to\ker(\Phi_N)\to\pi_1(N,p)\to\Mod(M_p)\to\Mod(M)\to 1.
\]
\end{proposition}
\begin{proof}
By Lemma \ref{changemonodromylemma}, we may assume that $\omega$ is the identity, and so $M\cong N\times\crcle$. 
Express $M$ as $(N\times[0,1])/\!\sim$, where $(x,0)\sim(x,1)$ for $x\in N$.
Define $\rho\colon[0,1]\to M$ by $\rho(t)=(p,t)$.
Then $\rho$ is a simple closed curve based at $p$, and $\pi_1(M,p)\cong\pi_1(N,p)\times\langle \rho\rangle$.
By Lemma \ref{pushroundfibrelemma}, $\Phi_M(\rho)$ is the identity map.
Therefore $\rho\in\ker(\Phi_M)$, and $\Phi_M(\pi_1(M,p))=\Phi_M(\pi_1(N,p))$.

Now assume $N$ is a surface. Let $\rho_N$ be a loop in $\Int(N)$ based at $p$ such that $\rho_N\in\ker(\Phi_M)$.
By Lemma \ref{mustbeperiodiclemma}, it follows that $\rho_N\in\ker(\Phi_N)$. 
Hence $\ker(\Phi_M)\cap\pi_1(N,p)\leq\ker(\Phi_N)$.
Note that if $\chi(N)\neq 0$ then $\ker(\Phi_N)\cong 1$, so this gives an exact sequence
\[
1\to\pi_1(N,p)\to\Mod(M_p)\to\Mod(M)\to 1.
\]

Conversely, let $\rho_N$ be a loop in $\Int(N)$ based at $p$ such that $\rho_N\in\ker(\Phi_N)$.
Then there is an isotopy $H_N\colon(N\times[0,1],\partial N\times[0,1])\to(N,\partial N)$ from the identity on $N$ to itself that can be used to define $\Phi_N(\rho_N)$.
Define $H_M\colon(M\times[0,1],\partial M\times[0,1])\to(M,\partial M)$ by
$H_M(x,s,t)=(H_N(x,t),s)$ for $x\in N$, $s\in[0,1)$ and $t\in[0,1]$.
Then $H_M$ is an isotopy from the identity on $M$ to itself that can be used to define $\Phi_M(\rho_N)$.
Thus $\rho_N\in\ker(\Phi_M)$.
Hence $\ker(\Phi_M)\cap\pi_1(N,p)=\ker(\Phi_N)$.
\end{proof}

\begin{corollary}\label{chi0cor}
If $M$ is an annulus or a torus then there is an exact sequence
\[
1\to\Mod(M_p)\to\Mod(M)\to 1.
\]
If $M$ is a M\"obius band then there is an exact sequence
\[
1\to\mathbb{Z}_2\to\Mod(M_p)\to\Mod(p)\to 1.
\]
If $M$ is a Klein bottle then there is an exact sequence
\[1\to\mathbb{Z}\rtimes\mathbb{Z}_2\to\Mod(M_p)\to\Mod(M)\to 1.\]
\end{corollary}
\begin{proof}
If $M$ is an annulus then it is fibred over a closed interval with monodromy the identity.

If $M$ is a torus then $\pi_1(M,p)\cong\mathbb{Z}^2$.
The proof of Proposition \ref{productbundleprop} shows that the generators of this group are in $\ker(\Phi)$.

If $M$ is a M\"obius band then it is fibred with base space a closed interval and monodromy a reflection. 
Let $\rho$ be a simple closed curve in $\Int(M)$ based at $p$ running once around the fibration. Then $\rho$ generates $\pi_1(M,p)$.
By Lemma \ref{pushroundfibrelemma}, $\phi_{\rho^2}$ acts as the identity on $M$. Hence $\rho^2\in\ker(\Phi)$.
On the other hand, the loop $\rho$ is orientation-reversing, so by Lemma \ref{reverseorientationlemma} we know that $\rho\notin\ker(\Phi)$.
Thus $\pi_1(M,p)/\ker(\Phi)\cong\mathbb{Z}_2$.

Now let $M$ be a Klein bottle. Then $M$ is fibred with base space a circle and monodromy a reflection.
Let $\rho$ and $\sigma$ be simple closed curves in $M$ based at $p$ such that $\rho$ runs once around the fibration and $\sigma$ runs once around the base circle.
Then every element of $\pi_1(M,p)$ can be expressed as $\sigma^s\cdot\rho^r$ for some $r,s\in\mathbb{Z}$.

Suppose that $\sigma^s\cdot\rho^r\in\ker(\Phi)$. 
Then there is an isotopy $H\colon M\times[0,1]\to M$ from the identity to itself that takes $p$ once around $\rho^r\cdot\sigma^s$.
This shows that, relative to its endpoints, $\rho$ is homotopic to $(\sigma^s\cdot\rho^r)\cdot\rho\cdot(\rho^{-r}\cdot\sigma^{-s})\simeq\sigma^s\cdot\rho\cdot\sigma^{-s}\simeq\sigma^{2s}\cdot\rho$.
However, in $\pi_1(M,p)$ this is only true if $s=0$.

On the other hand, as for a M\"obius band, $\rho\notin\ker(\Phi)$ but $\rho^2\in\ker(\Phi)$.
Therefore $\ker(\Phi)=\langle \rho^2\rangle\cong\mathbb{Z}$ and
$\pi_1(M,p)/\ker(\Phi)\cong\mathbb{Z}\rtimes\mathbb{Z}_2$.
\end{proof}

\begin{corollary}\label{productscor}
If $M\cong\sphere\times\crcle$ or $M$ is a solid torus then there is an exact sequence
\[
1\to\Mod(M_p)\to\Mod(M)\to 1.
\]
If $M\cong\sphere\widetilde{\times}\crcle$ or $M$ is a solid Klein bottle then there is an exact sequence
\[
1\to\mathbb{Z}_2\to\Mod(M_p)\to\Mod(M)\to 1.
\]
\end{corollary}
\begin{proof}
If $M$ is orientable then it is fibred over either a sphere or a disc with monodromy the identity. In either case, Proposition \ref{productbundleprop} gives the desired result.

If $M$ is non-orientable then it is fibred over either a sphere or a disc with monodromy a reflection.
Let $\rho$ be a loop in $\Int(M)$ based at $p$ that generates $\pi_1(M,p)\cong\mathbb{Z}$.
By Lemma \ref{reverseorientationlemma}, $\rho\notin\ker(\Phi)$.
On the other hand, Lemma \ref{pushroundfibrelemma} shows that $\rho^2\in\ker(\Phi)$.
Therefore $\ker(\Phi)\cong\mathbb{Z}$ and $\pi_1(M,p)/\ker(\Phi)\cong\mathbb{Z}_2$.
\end{proof}

\begin{lemma}\label{toruslemma}
Let $N=\mathbb{R}^2/\!\sim$, where $(a,b)\sim(a+m,b+m')$ for $m,m'\in\mathbb{Z}$ and $a,b\in\mathbb{R}$.
Let $\psi\colon N\to N$ be a homeomorphism.
Then there exists a matrix $\mathbf{A}\in \SL_2(\mathbb{Z})$ such that $\psi$ is isotopic to the map $\psi'\colon N\to N$ given by $\psi'(a,b)=\mathbf{A}(a,b)$ for $a,b\in\mathbb{R}$.
\end{lemma}
\begin{proof}
By Corollary \ref{fixpointcor}, we may assume that $\psi(0,0)=(0,0)$.
Then $\psi$ induces an isomorphism $\psi_*\colon\pi_1(N,(0,0))\to\pi_1(N,(0,0))$.
Since $\pi_1(N,(0,0))\cong\mathbb{Z}^2$, this isomorphism is defined by a matrix $\mathbf{A}\in \SL_2(\mathbb{Z})$.
Let $a_1,a_2,a_3,a_4\in\mathbb{Z}$ such that
\[
\mathbf{A}=\left(
\begin{array}{cc}
a_1&a_2\\
a_3&a_4
\end{array}
\right).
\]
Define $\psi'$ using $\mathbf{A}$.

Define $\rho_1,\rho_2\colon[0,1]\to N$ by $\rho_1(t)=(t,0)$ and $\rho_2(t)=(0,t)$.
The path $t\mapsto \psi(\rho_1(t))$ is a simple closed curve in $N$ based at $p$, and is homotopic, keeping $p$ fixed, to $\rho_1^{a_1}\cdot\rho_2^{a_3}$.
Define $\sigma_1\colon[0,1]\to N$ by $\sigma_1(t)=(a_1 t,a_3 t)$.
Then $\sigma_1$ is a simple closed curve and is homotopic, keeping $(0,0)$ fixed, to $\rho_1^{a_1}\cdot\rho_2^{a_3}$.
By \cite{MR0214087} Theorem 4.1, there is an isotopy of $N$, keeping $(0,0)$ fixed, from $\psi$ to a homeomorphism $\psi''\colon N\to N$ such that $\psi''(t,0)=\sigma_1(t)=\psi'(t,0)$ for $t\in[0,1]$.

Define $\sigma_2\colon[0,1]\to N$ by $\sigma_2(t)=(a_2 t,a_4 t)$.
The path $t\mapsto\psi''(0,t)$ is a simple closed curve in $N$ based at $(0,0)$ that only meets $\sigma_1$ at $(0,0)$ and is homotopic, keeping $\sigma_1$ fixed, to $\sigma_2$.
By Theorem \ref{homotopicarcsthm}, there is an isotopy of $N$, keeping $\sigma_1$ fixed, from $\psi''$ to a homeomorphism $\psi'''\colon N\to N$ such that $\psi'''(0,t)=\sigma_2(t)=\psi'(t)$ for $t\in[0,1]$.

Finally, Theorem \ref{alextrickthm} shows that $\psi'''$ is isotopic to $\psi'$ keeping $\sigma_1$ and $\sigma_2$ fixed.
\end{proof}

\begin{corollary}
Suppose that $M$ is fibred over a torus $N$ with monodromy $\omega\colon N\to N$.
Express $N$ as $\mathbb{R}^2/\!\sim$, where $(a,b)\sim(a+m,b+m')$ for $m,m'\in\mathbb{Z}$ and $a,b\in\mathbb{R}$, and
let $\mathbf{A}$ be the matrix given by Lemma \ref{toruslemma}.
Express $M$ as $(N\times[0,1])/\!\sim_{\omega}$, where $(x,1)\sim_{\omega}(\omega(x),0)$ for $x\in N$.
In addition, assume that $p=(0,0,0)$.
Define $\rho_1,\rho_2,\rho_{\omega}\colon[0,1]\to M$ by, for $t\in[0,1]$,
\[
\rho_1(t)=(t,0,0),\quad \rho_2(t)=(0,t,0),\quad\rho_{\omega}(t)=(0,0,t).
\]

If $\rho_{\omega}^{m_3}\cdot\rho_2^{m_2}\cdot\rho_1^{m_1}\in\ker(\Phi_M)$ for some $m_1,m_2,m_3\in\mathbb{Z}$ then
$\mathbf{A}^{m_3}=1$.
Conversely, if $\mathbf{A}^{m_3}=1$ for some $m_3\in\mathbb{Z}$ then $\rho_{\omega}^{m_3}\in\ker(\Phi_M)$.

In addition, given $m_1,m_2\in\mathbb{Z}$, $\sigma_2^{m_2}\cdot\sigma_1^{m_1}\in\ker(\Phi_M)$ if and only if $(m_1,m_2)$ is an eigenvector of $\mathbf{A}$ with eigenvalue $1$.
\end{corollary}
\begin{proof}
By Lemma \ref{changemonodromylemma}, we may assume that $\omega(a,b)=\mathbf{A}(a,b)$ for $a,b\in\mathbb{R}$.
Therefore, for $m\in\mathbb{Z}$, $\omega^m$ is the identity if and only if $\mathbf{A}^m=1$.
Lemmas \ref{mustbeperiodiclemma} and \ref{pushroundfibrelemma} give the first two conclusions.

Suppose $\rho_2^{m_2}\cdot\rho_1^{m_1}\in\ker(\Phi_M)$ for some $m_1,m_2\in\mathbb{Z}$.
By Lemma \ref{mustbeperiodiclemma}, $\omega(\rho_2^{m_2}\cdot\rho_1^{m_1})$ is homotopic to $\rho_2^{m_2}\cdot\rho_1^{m_1}$ in $N$ keeping $(0,0)$ fixed.
On the other hand, $\omega(\rho_2^{m_2}\cdot\rho_1^{m_1})$ is homotopic to $\rho_2^{m'_2}\cdot\rho_1^{m'_1}$ where $(m'_1,m'_2)=\mathbf{A}(m_1,m_2)$.
Therefore $(m'_1,m'_2)=(m_1,m_2)$ and $(m_1,m_2)$ is an eigenvector of $\mathbf{A}$ with eigenvalue $1$.

Suppose instead that $(m_1,m_2)$ is an eigenvector of $\mathbf{A}$ with eigenvalue $1$ for some $m_1,m_2\in\mathbb{Z}$.
Define $H_N\colon N\times[0,1]\to N$ by
$H_N(a,b,t)=(a+m_1 t,b+m_2 t)$ for $a,b\in\mathbb{R}$ and $t\in[0,1]$.
Further define $H\colon M\times[0,1]\to M$ by
$H(a,b,s,t)=(H_N(a,b,t),s)$ for $a,b\in\mathbb{R}$, $s\in[0,1)$ and $t\in[0,1]$.
For $a,b\in \mathbb{R}$ and $t\in[0,1]$, 
\begin{align*}
H(a,b,1,t)&=(a+m_1 t,b+m_2 t,1)\\
&\sim (\mathbf{A}(a+m_1 t,b+m_2 t),0)\\
&=(\mathbf{A}(a,b)+t(m_1,m_2),0)\\
&=(H_N(\mathbf{A}(a,b),t),0)\\
&=H(\omega(a,b),0,t).
\end{align*}
Hence $H$ is well-defined.
In addition, for $a,b\in\mathbb{R}$ and $s\in[0,1)$,
\[
H(a,b,s,1)=(H_N(a,b,1),s)=(a+m_1,b+m_2,s).
\]
Thus $H$ is an isotopy from the identity on $M$ to itself.
Define $\rho\colon[0,1]\to M$ by $\rho(t)=H(0,0,0,t)$ for $t\in[0,1]$.
Then $\rho\in\ker(\Phi_M)$.
Since, for $t\in[0,1]$,
\[
\rho(t)=(H_N(0,0,t),0)=(m_1 t,m_2 t,0),
\]
$\rho$ is homotopic in $N$, keeping $(0,0)$ fixed, to $\rho_2^{m_2}\cdot\rho_1^{m_1}$.
Therefore $\rho_2^{m_2}\cdot\rho_1^{m_1}\in\ker(\Phi_M)$.
\end{proof}

\begin{corollary}\label{torusbundlecor}
Suppose that $M$ is fibred over a torus $N$ with monodromy $\omega\colon N\to N$.
Express $N$ as $\mathbb{R}^2/\!\sim$, where $(a,b)\sim(a+m,b+m')$ for $a,b\in\mathbb{R}$ and $m,m'\in\mathbb{Z}$.
Assume that
$\omega(a,b)=\mathbf{A}(a,b)$ for $a,b\in\mathbb{R}$, where $\mathbf{A}\in\SL_2(\mathbb{Z})$.
Then there is an exact sequence
\[
1\to K\to\pi_1(M,p)\to\Mod(M_p)\to\Mod(M)\to 1,
\]
where $K$ is as follows.
\begin{itemize}
\item If $\mathbf{A}$ is the identity matrix then $K=\mathbb{Z}^3$.
\item If $\mathbf{A}$ is not the identity matrix but $\tr(\mathbf{A})=2$ then $K=\mathbb{Z}$.
\item If $-\mathbf{A}$ is the identity matrix then $K=\mathbb{Z}$.
\item If $\tr(\mathbf{A})\in\{0,\pm 1\}$ then $K=\mathbb{Z}$.
\item Otherwise $K=1$.
\end{itemize}
\end{corollary}

\section{Non-fibred $3$--manifolds}\label{nonfibredsection}
%master document is splitlinks.tex

\begin{proposition}\label{nottorusprop}
Suppose $M$ is a $3$--manifold, that $\partial M\neq\emptyset$, and that at least one component of $\partial M$ is neither a torus nor a Klein bottle.
Then there is an exact sequence
\[
1\to\pi_1(M,p)\to \Mod(M_p)\to\Mod(M)\to 1.
\]
\end{proposition}
\begin{proof}
Let $S$ be a boundary component of $M$ that is neither a torus nor a Klein bottle.
Then $S$ is a closed, connected surface with $\chi(S)\neq 0$.
Let $\rho$ be a simple closed curve in $\Int(M)$ based at $p$ such that $\rho\in\ker(\Phi_M)$.
Then there is an isotopy
$H\colon (M\times[0,1],\partial M\times[0,1])\to(M,\partial M)$ from the identity on $M$ to itself that can be used to define $\Phi_M(\rho)$.

Pick a point $p'$ on $S$.
Let $H_S\colon(S\times[0,1],\partial S\times[0,1])\to(S,\partial S)$ be the restriction of $H$ to $S$.
Then $H_S$ gives an isotopy of $S$ from the identity to itself.
Let $\rho'\colon[0,1]\to S$ be the path followed by $p'$ under $H_S$.
Then $H_S$ can be used to define $\Phi_S(\rho')$ on $S$.
By Proposition \ref{surfacesprop}, this means $\rho'$ is null-homotopic in $S$.
Lemma \ref{fixbdypointlemma} therefore says that $\rho$ is null-homotopic in $M$.
Hence $\ker(\Phi_M)\cong 1$.
\end{proof}

\begin{lemma}[\cite{MR0224099} Lemma 7.4]\label{swapsideslemma}
Let $M'$ be a closed, irreducible $3$--manifold, and let $R\subseteq M'$ be an incompressible surface. Let $\psi\colon M'\to M'$ be a homeomorphism with $\psi(R)=R$.
If $\psi$ is homotopic to the identity then $\psi$ does not switch the sides of $R$.
\end{lemma}

Imitating the proof of Lemma \ref{swapsideslemma} gives the following.

\begin{lemma}\label{notswitchsideslemma}
Let $S$ be a separating $2$--sphere in a compact, connected $3$--manifold $M'$ other than a $3$--sphere.
Let $\psi\colon M'\to M'$ be a homeomorphism such that $\psi(S)=S$.
Suppose that $\psi$ is isotopic to the identity. Then $\psi$ does not switch the sides of $S$. 
\end{lemma}
\begin{proof}
Let $M_1$ and $M_2$ be the two pieces into which $S$ divides $M'$. That is, $M_1$ and $M_2$ are the closures of the two components of $M'\setminus S$.
First suppose that $\partial M_1\setminus S\neq\emptyset$. Since any isotopy of $M'$ keeps $\partial M'$ invariant, $\psi(\partial M_1\setminus S)\subseteq M_1$. Hence $\psi$ does not switch the sides of $S$.
Suppose next that $M_1$ is a $3$--ball. As $M'$ is not $\Sphere$, this means $M_2$ is not a $3$--ball. Thus, since $M_1\ncong M_2$, again $\psi$ cannot switch the sides of $S$.

Now assume that $M_1$ and $M_2$ each have no boundary other than $S$, and that neither is a $3$--ball. By the Poincar\'e Conjecture this means that $\pi_1(M_1)\ncong 1\ncong\pi_1(M_2)$.
Let $p'$ be a point on $S$.
Then $\pi_1(M',p')\cong\pi_1(M_1,p')*\pi_1(M_2,p')$.
If $\psi$ switches $M_1$ and $M_2$, there is an inner automorphism of $\pi_1(M',p')$ (given by the path travelled by $p'$ under the isotopy) that interchanges $\pi_1(M_1,p')$ and $\pi_1(M_2,p')$.
Express the corresponding element as
$x_1\cdot y_1\cdots x_m\cdot y_m$, where $x_i\in \pi_1(M_1,p')$ and $y_i\in\pi_1(M_2,p')$ for $1\leq i\leq m$, and
$x_{i+1}\neq 1\neq y_i$ for $1\leq i<m$.
Without loss of generality we may assume that $x_1\neq 1$.
Let $y\in\pi_1(M_2,p)\setminus\{1\}$.
Then by assumption $(x_1\cdot y_1\cdots x_m\cdot y_m)^{-1}\cdot y\cdot (x_1\cdot y_1\cdots x_m\cdot y_m)\in\pi_1(M_1,p')$.
This is a contradiction, since $y_m^{-1}\cdot x_m^{-1}\cdots y_1^{-1}\cdot x_1^{-1}\cdot y\cdot x_1\cdot y_1\cdots x_m\cdot y_m\notin\pi_1(M_1,p')$.
Therefore $\psi$ does not switch the sides of $S$.
\end{proof}

The following lemma follows from the proof of \cite{MR0314054} Theorem 1.1 together with the Poincar\'e Conjecture.

\begin{lemma}\label{homotopicsphereslemma}
Let $M'$ be a $3$--manifold, and let $S_1$ and $S_2$ be $2$--spheres disjointly embedded in $M'$.
If $S_1$ is homotopically trivial then it bounds a $3$--ball.
If neither $S_1$ nor $S_2$ bounds a $3$--ball but
they are homotopic, then they co-bound a manifold homeomorphic to $\sphere\times[0,1]$.
\end{lemma}

\begin{remark}\label{irreducibleremark}
Together with the Sphere Theorem (see, for example, \cite{MR0415619} Theorem 4.3), this implies that an orientable $3$--manifold $M'$ is irreducible if and only if $\pi_2(M')\cong 1$.
Since $\pi_2(\widetilde{M}')=\pi_2(M')$ for any cover $\widetilde{M}'$ of $M'$ (see, for example, \cite{MR1867354} Proposition 4.1), this implies that $\widetilde{M}'$ is irreducible if and only if $M'$ is.
\end{remark}

\begin{theorem}[\cite{MR0346792} Corollary 1]\label{quotientthm}
If $\sphere\times\crcle$ is a covering space of a manifold $M'$ then $M'$ is either 
$\sphere\times\crcle$, $\sphere\widetilde{\times}\crcle$,
$\mathbb{R}P^2\times\crcle$ or $\mathbb{R}P^3\#\mathbb{R}P^3$.
\end{theorem}

\begin{proposition}\label{nonprimeprop}
Suppose $M$ is a $3$--manifold and contains an essential separating $2$--sphere.
Assume $M$ is not $\mathbb{R}P^3\#\mathbb{R}P^3$.
Then there is an exact sequence
\[1\to\pi_1(M,p)\to\Mod(M_p)\to\Mod(M)\to 1.\]
\end{proposition}
\begin{proof}
Let $S$ be an essential separating $2$--sphere in $M$, and isotope $S$ in $M$ so that $p$ lies on $S$.
Suppose that $\ker(\Phi)\ncong 1$.
Let $\rho\colon[0,1]\to M$ be a loop based at $p$ such that $\rho\in\ker(\Phi)\setminus\{1\}$.
Homotope $\rho$, keeping its endpoints fixed, to be in general position with respect to $S$ and to have minimal intersection with $S$.
Further homotope $\rho$ so that it only meets $S$ at $p$.
Then the points where $\rho$ meets $S$ divide it into loops $\rho_1,\ldots,\rho_m$ based at $p$.
Since $|\rho\cap S|$ is minimal, for $1\leq i\leq j\leq m$ the loop $\rho_i\cdots\rho_j$ is not null-homotopic in $M$.

Let $(\widetilde{M},\widetilde{p})$ be the universal cover of $(M,p)$. 
Let $\widetilde{S}$ be the lift of $S$ containing $\widetilde{p}$, and let $\widetilde{\rho}\colon[0,1]\to\widetilde{M}$ be the lift of $\rho$ that begins at $\widetilde{p}$.
Because no sub-loop of $\rho$ is null-homotopic, $\widetilde{\rho}$ does not meet any lift of $S$ more than once.
Let $\widetilde{S}'$ be the lift of $S$ containing $\widetilde{\rho}(1)$.
Then $\widetilde{S}'$ is distinct from $\widetilde{S}$.

Since $\rho\in\ker(\Phi)$, there is an isotopy $H_{\rho}\colon (M\times[0,1],\partial M\times[0,1])\to(M,\partial M)$ from the identity on $M$ to itself that can be used to define $\phi_{\rho}$.
This lifts to an isotopy of $\widetilde{M}$ that takes $\widetilde{S}$ to $\widetilde{S}'$.

Lemma \ref{homotopicsphereslemma} shows that $\widetilde{S}$ and $\widetilde{S}'$ together bound a manifold $N\cong\sphere\times[0,1]$ in $\widetilde{M}$.
As $\widetilde{\rho}$ connects the two boundary components of $N$, and $\widetilde{M}$ is simply connected, $N$ contains $\widetilde{\rho}$.
Every lift of $S$ that $\widetilde{\rho}$ passes through therefore lies in $N$, and together these divide $N$ into $m$ copies of $\sphere\times[0,1]$.
%Take a sphere in S^2\times[0,1]. Fill one end with a ball. Then we have a sphere in a 3-ball, so it bounds a ball. By annulus theorem (Rourke--Sanderson 3.19), the piece between the sphere and the boundary is S^2\times[0,1]. Now remove the added ball. If we don't remove it from inside the sphere then the sphere bounds a ball, and once the arc runs into this ball it can't get out again without crossing it again. So we must remove the ball from inside the sphere. Then the piece between the sphere and the boundary of the ball is also S^2\times[0,1].

Suppose $m>1$. Then $\rho_1\cdot\rho_2$ forms a closed curve in $M$ that meets $S$ transversely twice. 
The lift of $\rho_1\cdot\rho_2$ to $\widetilde{M}$ ends on a lift $\widetilde{S}_2$ of $S$ contained in $N$.
Thus $\widetilde{S}$ and $\widetilde{S}_2$ also bound a manifold homeomorphic to $\sphere\times[0,1]$.
Let $G_2$ be the cyclic subgroup of $\pi_1(M,p)$ generated by $\rho_1\cdot\rho_2$.
Quotienting $\widetilde{M}$ by the action of $G_2$ glues together $\widetilde{S}$ and $\widetilde{S}_2$, giving a manifold homeomorphic to $\sphere\times\crcle$.
Hence $M$ is a quotient of $\sphere\times\crcle$, and Theorem \ref{quotientthm} applies.
Recall that by assumption $M$ is not $\mathbb{R}P^3\#\mathbb{R}P^3$, and also $M$ is not prime. 
Both $\sphere\times\crcle$ and $\sphere\widetilde{\times}\crcle$ are prime (see \cite{MR0415619} 3.12).
Since $\pi_1(\mathbb{R}P^2\times\crcle)\cong\mathbb{Z}_2\times\mathbb{Z}$, which is abelian and therefore not a non-trivial free product, $\mathbb{R}P^2\times\crcle$ is also prime (see \cite{MR0415619} 3.11 and note that the Poincar\'e conjecture shows that there are no fake $3$--balls).
Thus this case cannot occur.

Now suppose that $m=1$. 
Let $M_1$ and $M_2$ be the closures of the two components of $M\setminus S$. 
Then, under the covering map, $N$ projects to either $M_1$ or $M_2$.
Without loss of generality, assume this is $M_1$.
By doubling $N$ along its boundary to $\sphere\times\crcle$, we may therefore define a covering map from $\sphere\times\crcle$ to $M_1\# M_1$.
Again applying Theorem \ref{quotientthm} we find that $M_1\# M_1$ is homeomorphic to $\mathbb{R}P^3\#\mathbb{R}P^3$.
Because $\mathbb{R}P^3$ is prime, 
by \cite{MR0415619} Theorem 3.21 this means that $M_1\cong\mathbb{R}P^3$.
Therefore we know that $\rho$ is the only non-trivial element of $\pi_1(M_1,p)$, which has order $2$. 

Every element of $\pi_1(M,p)$ can be expressed as $\rho^{\varepsilon_1}\cdot (\sigma_1\cdot\rho)\cdots(\sigma_r\cdot\rho)\cdot\rho^{\varepsilon_2}$, where $r\in\mathbb{Z}_{\geq 0}$, $\sigma_i\in\pi_1(M_2,p)\setminus\{1\}$, and $\varepsilon_1,\varepsilon_2\in\{0,1\}$.
Let $G_1$ be the subgroup of $\pi_1(M,p)$ consisting of elements for which $r+\varepsilon_1+\varepsilon_2$ is even.
Then $G_1$ has index $2$ in $\pi_1(M,p)$, with $\{1,\rho\}$ as a set of coset representatives. 
In particular, $\rho\notin G_1$.

Let $(\widehat{M},\widehat{p})$ be the double cover of $(M,p)$ with $\pi_1(\widehat{M},\widehat{p})\cong G_1$.
Let $\widehat{S}$ be the lift of $S$ containing $\widehat{p}$.
Let $\widehat{S}'$ be the other lift of $S$ and $\widehat{p}'$ the lift of $p$ in $\widehat{S}'$.
As before we find that $\widehat{S}$ and $\widehat{S}'$ are homotopic in $\widehat{M}$ and so bound a manifold $\widehat{N}\cong\sphere\times[0,1]$.

Let $\widehat{\rho}$ be the lift of $\rho$ that begins at $\widehat{p}$.
If $\widehat{\rho}$ is not contained in $\widehat{N}$ then $\widehat{N}$ projects onto $M_2$ and we again find that $M\cong\mathbb{R}P^3\#\mathbb{R}P^3$.
Assume instead that $\widehat{\rho}$ is contained in $\widehat{N}$.
Suppose $\widehat{M}\setminus\widehat{N}$ is connected. Then there is a path $\widehat{\sigma}$ in $\widehat{M}\setminus\widehat{N}$ connecting $\widehat{p}'$ to $\widehat{p}$.
This projects to a loop $\sigma$ in $M_2$ based at $p$. Then $\rho\cdot\sigma\notin G_1$. This contradicts that $\widehat{\rho}\cdot\widehat{\sigma}$ forms a closed loop in $\widehat{M}$.
Hence $\widehat{M}\setminus\widehat{N}$ has two components, and $\widehat{S}$ is separating in $\widehat{M}$. 

As before, $H_{\rho}$ lifts to an isotopy of $\widehat{M}$ from the identity to a map $\psi\colon\widehat{M}\to\widehat{M}$.
Moreover, $\psi(\widehat{S})=\widehat{S}'$, $\psi(\widehat{S}')=\widehat{S}$ and $\psi(\widehat{N})=\widehat{N}$, since $\widehat{M}\setminus\widehat{N}$ is disconnected. 
Let $\widehat{H}'\colon(\widehat{M}\times[0,1],\partial\widehat{M}\times[0,1])\to(\widehat{M},\partial\widehat{M})$ be an isotopy that pushes $\widehat{S}'$ across the product structure of $\widehat{N}$ to $\widehat{S}$.
This gives an isotopy from $\psi$ to 
a map $\psi'\colon\widehat{M}\to\widehat{M}$ with $\psi'(\widehat{S})=\widehat{S}$.
By construction, $\psi'$ switches the sides of $\widehat{S}$ in $\widehat{M}$.
However, $\psi'$ is isotopic to the identity on $\widehat{M}$, contradicting Lemma \ref{notswitchsideslemma}.

Therefore $\ker(\Phi)\cong 1$. 
\end{proof}

\begin{lemma}\label{rp3lemma}
If $M=\mathbb{R}P^3\# \mathbb{R}P^3$ then there is an exact sequence
\[
1\to\pi_1(M,p)\to\Mod(M_p)\to\Mod(M)\to 1.
\]
\end{lemma}
\begin{proof}
Let $S$ be an essential separating $2$--sphere in $M$, and isotope $S$ so that $p$ lies on $S$.
Let $M_1$ and $M_2$ be the closures of the two components of $M\setminus S$.
Recall that $\pi_1(\mathbb{R}P^3)\cong\mathbb{Z}_2$, so $\pi_1(M,p)\cong\mathbb{Z}_2 *\mathbb{Z}_2$.
Let $x$ and $y$ be simple closed curves in $M_1$ and $M_2$ respectively based at $p$ such that $\pi_1(M_1,p)=\{1,x\}$  and $\pi_1(M_2,p)=\{1,y\}$.
Then every element of $\pi_1(M,p)$ can be expressed uniquely as $x^{\varepsilon_1}\cdot(y\cdot x)^r\cdot y^{\varepsilon_2}$ for some $r\in\mathbb{Z}_{\geq 0}$ and $\varepsilon_1,\varepsilon_2\in\{0,1\}$.

Let $\rho$ be a loop in $M$ based at $p$ such that $\rho\in\ker(\Phi_M)\setminus\{1\}$.
Without loss of generality, assume $\rho=(y\cdot x)^{r_{\rho}}\cdot y^{\varepsilon_{\rho}}$, where either $r_{\rho}\geq 1$ or $\varepsilon_{\rho}=1$.

We now build a finite cover of $M$.
Set $m=2r_{\rho}+\varepsilon_{\rho}$. Then $m>0$.
Let $N_1$ and $N_{m+3}$ be copies of $M_1$, and let $N_2,\ldots,N_{m+2}$ be copies of $\sphere\times[0,1]$.
Note that $\sphere\times[0,1]$ double covers $M_1$.
For $1\leq i\leq m+2$, glue $N_i$ to $N_{i+1}$ along one boundary component of each, forming a closed $3$--manifold $\widehat{M}$.
Denote by $S_i$ the copy of $\sphere$ between $N_i$ and $N_{i+1}$ for $1\leq i\leq m+2$.
Let $\widehat{p}$ be the copy of $p$ in $N_1$.
We may define a covering map $\pi\colon(\widehat{M},\widehat{p})\to (M,p)$ by specifying that
$N_i$ covers $M_1$ if $i$ is odd and covers $M_2$ if $i$ is even.

Let $\widehat{\rho}\colon[0,1]\to\widehat{M}$ be the lift of $\rho$ beginning at $\widehat{p}$.
Then $\widehat{\rho}(1)$ lies on $S_{m+1}$.
Let $H\colon M\times[0,1]\to M$ be an isotopy from the identity on $M$ to itself that can be used to define $\Phi_M(\rho)$.
This lifts to an isotopy $\widehat{H}\colon\widehat{M}\times[0,1]\to \widehat{M}$ from the identity on $\widehat{M}$ to a covering transformation $\widehat{\psi}\colon\widehat{M}\to\widehat{M}$.
As $\widehat{\psi}(\widehat{p})=\widehat{\rho}(1)$, we see that $\widehat{\psi}(N_1)$ is either $N_{m+1}$ (if $\varepsilon_{\rho}=0$) or $N_{m+2}$ (if $\varepsilon_{\rho}=1$).
This contradicts that $N_0\cong M_1\ncong\sphere\times[0,1]\cong N_{m+1}\cong N_{m+2}$.

Hence $\ker(\Phi_M)\cong 1$.
\end{proof}

\begin{remark}
Given Proposition \ref{nottorusprop}, the proof of Lemma \ref{rp3lemma} could be extended to give a proof of Proposition \ref{nonprimeprop}, using that a $3$--manifold with a toral boundary component has a non-trivial fundamental group and that the fundamental group of a compact $3$--manifold is residually finite (see, for example, \cite{2012arXiv1205.0202A} C.25). However, the proof of Lemma \ref{hakenlemma} will follow the pattern of that of Proposition \ref{nonprimeprop}.
\end{remark}

\begin{remark}\label{solidtorusremark}
Suppose $M'$ is a connected, orientable, irreducible $3$--manifold with non-empty boundary consisting of tori. Then $M'$ has compressible boundary if and only if $M'$ is a solid torus.
\end{remark}

\begin{lemma}[\cite{MR565450} III.8]\label{pi1injectivelemma}
An orientable surface other than a $2$--sphere in an orientable $3$--manifold is incompressible if and only if it is $\pi_1$--injective.
\end{lemma}

\begin{proposition}[\cite{MR1315011} Proposition 4.8]\label{productregionprop}
Let $M'$ be an orientable, irreducible, $\partial$--irreducible, Haken $3$--manifold.
Let $R_1$ and $R_2$ be disjoint, incompressible and $\partial$--incompressible surfaces, other than $2$--spheres, properly embedded in $M'$. If $R_1$ and $R_2$ are isotopic then they co-bound a submanifold of $M'$ homeomorphic to $R_1\times[0,1]$.
\end{proposition}

\begin{proposition}[\cite{MR565450} VI.25]\label{cyclicsubgroupprop}
Let $M'$ be a compact, orientable, Haken $3$--manifold other than a solid torus, and let $G$ be an infinite cyclic normal subgroup of $\pi_1(M')$. 
Suppose there exists an orientable, incompressible, $\partial$--incompressible surface $R$ properly embedded in $M'$ with $G\leq\pi_1(R)$.
Then $R$ is an annulus or a torus, and $M$ is homeomorphic to a Seifert fibred manifold via a homeomorphism taking $R$ to a union of fibres, and $G$ is a subgroup of the infinite cyclic group generated by the path around a regular fibre.
\end{proposition}

\begin{proposition}\label{alltoriprop}
Suppose $M$ is an orientable, irreducible, $\partial$--irreducible $3$--manifold.
Further assume that $M$ is not fibred, but has non-empty boundary consisting of tori.
Then either there is an exact sequence
\[1\to\pi_1(M,p)\to\Mod(M_p)\to\Mod(M)\to 1,\] 
or $M$ is Seifert fibred and $\ker(\Phi)$ is generated by a power of a regular fibre.
\end{proposition}
\begin{proof}
As $M$ is an orientable, irreducible $3$--manifold other than a $3$--ball,
by \cite{MR0224099} Lemma 1.1.6 there exists an oriented, incompressible surface $R$ properly embedded in $M$ with non-empty boundary that is not null-homologous in $\partial M$.
Boundary compressing $R$ will not change the homology class of $\partial R$, so we may assume $R$ is $\partial$--incompressible. We may also assume that $R$ is connected.

Let $S$ be a component of $\partial M$ such that $R\cap S$ is not null-homologous in $S$.
Suppose some curve of $R\cap S$ is inessential in $S$. As $R$ is incompressible, an innermost such curve bounds a disc component of $R$. We have assumed that $R$ is connected, so this implies that $\partial R$ is a single inessential loop in $S$. This contradicts that $\partial R$ is not null-homologous in $\partial M$.
Hence every curve of $R\cap S$ is essential in $S$.
Isotope $R$ so that $p$ lies on $R$. Fix $p'\in R\cap S$, and choose a path $\sigma$ in $R$ from $p'$ to $p$.

Define $\theta\colon \pi_1(M,p)\to\mathbb{Z}$ by taking the algebraic intersection with $R$.
Let $G=\ker(\theta)$, and let $(\widetilde{M},\widetilde{p})$ be the infinite cyclic cover of $(M,p)$ with $\pi_1(\widetilde{M},\widetilde{p})\cong G$.
By Remark \ref{irreducibleremark}, $\widetilde{M}$ is irreducible. 
Since $\partial M$ is incompressible in $M$, using Lemma \ref{pi1injectivelemma} we see that $\partial \widetilde{M}$ is incompressible in $\widetilde{M}$. %Note that if a loop is null-homotopic in the infinite cover then it is also null-homotopic in some finite cover, which is compact.
Let $\widetilde{R}$ be the lift of $R$ containing $\widetilde{p}$.

Assume $\ker(\Phi_M)\neq 1$, and let $\rho$ be a simple closed curve in $\Int(M)$ based at $p$ such that $\rho\in\ker(\Phi_M)\setminus\{1\}$.
Then there is an isotopy $H\colon(M\times[0,1],\partial M\times[0,1])\to(M,\partial M)$ from the identity to itself that can be used to define $\Phi_M(\rho)$.
This lifts to an isotopy $\widetilde{H}\colon (\widetilde{M}\times[0,1],\partial\widetilde{M}\times[0,1])\to(\widetilde{M},\partial\widetilde{M})$ that begins at the identity.
This isotopy takes $\widetilde{R}$ to a lift $\widetilde{R}'$ of $R$.

Suppose that $\widetilde{R}'\neq\widetilde{R}$. Then, by Proposition \ref{productregionprop}, $\widetilde{R}$ and $\widetilde{R}'$ bound a sub-manifold $N$ of $\widetilde{M}$ homeomorphic to $R\times[0,1]$.
If $N$ contains any other lift of $R$, then each of these also co-bounds with $\widetilde{R}$ a manifold homeomorphic to $R\times[0,1]$ (see \cite{MR0224099} Corollary 3.2). By taking the lift of $R$ in $N\setminus\widetilde{R}$ closest to $\widetilde{R}$, we find that $M$ is given by quotienting $R\times[0,1]$ by some homeomorphism $\omega\colon R\times\{0\}\to R\times\{1\}$. 
This contradicts that $M$ is not fibred.

Therefore $\widetilde{R}'=\widetilde{R}$. 
This means that Lemma \ref{fixboundarylemma} applies, and we may now assume that $H(R\cap S,t)=R\cap S$ for $t\in[0,1]$. In particular, the path $\rho'$ of $p'$ under $H$ is contained within a single curve of $R\cap S$.

Let $\sigma'\colon[0,1]\to S$ be a path with $\sigma'(0)=\sigma'(1)=p'$ that runs once around the component of $R\cap S$ containing $p'$.
Then $\rho'$ is homotopic to $(\sigma')^r$ for some $r\in\mathbb{Z}$.
As in the proof of Lemma \ref{fixbdypointlemma}, $\sigma$ is homotopic relative to its endpoints to $(\rho')^{-1}\cdot\sigma\cdot\rho$, and so
\[
\rho\simeq \sigma^{-1}\cdot\rho'\cdot\sigma\simeq\sigma^{-1}\cdot(\sigma')^r\cdot\sigma.
\]

We therefore conclude that $\ker(\Phi_M)\leq \pi_1(R,p)$, and either $\ker(\Phi_M)\cong 1$ or $\ker(\Phi_M)\cong\mathbb{Z}$.
If $\ker(\Phi_M)\cong 1$ then there is an exact sequence
\[ 1\to\pi_1(M,p)\to\Mod(M_p)\to\Mod(M)\to 1.\]

Now assume $\ker(\Phi_M)\cong\mathbb{Z}$. Note that $\ker(\Phi_M)\trianglelefteq\pi_1(M,p)$, and $M$ is Haken and not a solid torus.
By Proposition \ref{cyclicsubgroupprop}, $R$ is an annulus and $M$ is Seifert fibred with $\ker(\Phi_M)$ a subgroup of the infinite cyclic group generated by the path around a regular fibre.
\end{proof}

\begin{lemma}
Suppose $M$ is an orientable $3$--manifold and is Seifert fibred with orientable base space. Let $\rho\colon[0,1]\to \Int(M)$ be a loop based at $p$ that runs once around a regular fibre. Then $\rho\in\ker(\Phi_M)$.
\end{lemma}
\begin{proof}
Suppose that $M$ is Seifert fibred with base orbifold a compact, connected, orientable surface $N$, and with $m$ exceptional fibres with coefficients $(r_1,s_1),\ldots,(r_m,s_m)$.
The idea of this proof is similar to that of Lemma \ref{pushroundfibrelemma}. We isotope each point on a regular fibre once around that fibre. However, to achieve a continuous map we must isotope each point on an exceptional fibre around the fibre $r_i$ times.

By an isotopy we may assume that $p$ lies on a regular fibre.
Let $M'$ be the manifold obtained from $M$ by removing an open fibred neighbourhood of each exceptional fibre.
Let $N'$ be the orientable surface obtained from $N$ by removing $m$ open discs.
Then $M'$ is an $\crcle$--bundle over $N'$.
As $M'$ is orientable, in fact $M'\cong N'\times\crcle$.
Express $M'$ as $(N'\times[0,1])/\!\sim_{N'}$, where $(x,0)\sim_{N'}(x,1)$ for $x\in N'$, such that each fibre is of the form $x\times[0,1]$ for some $x\in N'$.

Define $H'\colon (M'\times[0,1],\partial M'\times[0,1])\to (M',\partial M')$ by, for $x\in N'$, $s\in[0,1)$ and $t\in[0,1]$,
\[
H'(x,s,t)=
\begin{cases}
(x,s+t)&s+t<1,\\
(x,s+t-1)&s+t\geq 1.
\end{cases}
\]
Then $H'$ is an isotopy from the identity on $M'$ to itself that takes $p$ once around a regular fibre of $M$.
We wish to extend $H'$ to an isotopy of $M$.

Consider the $i$th exceptional fibre of $M$, for $1\leq i\leq m$.
Let $\nhd_i$ be the closure of the fibred neighbourhood of this fibre that was removed from $M$ to give $M'$.
Then $\nhd_i$ is homeomorphic to $(S\times[0,1])/\!\sim_i$, where $S$ is a disc, $\omega_i\colon (S,\partial S) \to(S,\partial S)$ is rotation by $2\pi s_i/r_i$ and $(x,1)\sim_i(\omega_i(x),0)$ for $x\in S$.
The product structure of $M'$ carries over to $\partial\!\nhd_i$, so that $\partial\!\nhd_i\cong\mathbb{R}^2/\!\sim$, where $(a+m_1,b+m_1)\sim(a,b)$ for $a,b\in\mathbb{R}$ and $m_1,m_2\in\mathbb{Z}$, and each fibre is of the form $\{a\}\times\mathbb{R}$ for some $a\in\mathbb{R}$. 
Let $H_i$ be the restriction of $H'$ to $\partial\!\nhd_i$.
Then $H_i(a,b,t)=(a,b+t)$ for $a,b\in[0,1)$ and $t\in[0,1]$.

Define $\sigma_i\colon[0,1]\to\partial\!\nhd_i$ by $\sigma_i(t)=(s_i t,r_i t)$ for $t\in[0,1]$.
Then $\sigma$ is a simple closed curve in $\partial \!\nhd_i$ and bounds a disc in $\nhd_i$ (note that here sign conventions are of no concern as they will not play a part in the proof).
We may assume that this disc is $S$.
The fibres of $\mathbb{R}^2/\!\sim$ match up with the Seifert fibres of $(S\times[0,1])/\!\sim_i$ in such a way that each fibre of $\mathbb{R}^2/\!\sim$ is the union of $r_i$ arcs of the form $\{x\}\times[0,1)$ in $(S\times[0,1])/\!\sim_i$.
It is now clear that we may extend $H_i$ to an isotopy of $\nhd_i$ from the identity to itself. This isotopy takes each point of $(S\times[0,1])/\!\sim_i$ around $r_i$ of these arcs. Away from the exceptional fibre of $M$, the $r_i$ arcs are distinct, but for the exceptional fibre these are all the central arc of $(S\times[0,1])/\!\sim_i$.

After extending $H'$ to $\nhd_1,\ldots,\nhd_m$ in this way, we have an isotopy of $M$ from the identity to itself that takes $p$ once around a regular fibre, as required.
\end{proof}

\begin{theorem}[\cite{2012arXiv1210.7298P} Theorem 1.1]\label{separablethm}
Let $M'$ be a compact, connected $3$--manifold, and let $R$ be a connected, $\pi_1$--injective surface properly embedded in $M'$. Then $\pi_1(R)$ is separable in $\pi_1(M')$. That is, $\pi_1(R)$ equals the intersection of the finite index subgroups of $\pi_1(M')$ containing it.
\end{theorem}

\begin{lemma}\label{hakenlemma}
Suppose that $M$ is a closed, orientable, irreducible $3$--manifold containing an embedded, orientable, connected, incompressible surface $R$ (with positive genus).
Then either $\ker(\Phi)\trianglelefteq\pi_1(R)$ or $M$ is finitely covered by an $R$--bundle over $\crcle$.
\end{lemma}
\begin{proof}
Isotope $R$ so that $p$ lies on $R$.
As $R$ is incompressible, $\pi_1(R,p)\leq\pi_1(M,p)$.
Suppose $\ker(\Phi_M)\nleq\pi_1(R,p)$, and
let $\rho$ be a simple closed curve in $M$ based at $p$ such that $\rho\in\ker(\Phi_M)\setminus\pi_1(R,p)$.
We may choose $\rho$ to be transverse to $R$.
By Theorem \ref{separablethm}, there exists a group $G\leq\pi_1(M,p)$
of finite index with $\pi_1(R,p)\leq G$ and $\rho\notin G$.
Let $(\widehat{M},\widehat{p})$ be the finite-sheeted cover of $(M,p)$ with $\pi(\pi_1(\widehat{M},\widehat{p}))= G$, where $\pi\colon\widehat{M}\to M$ is the covering map.
In addition, let $\widehat{\rho}$ be the lift of $\rho$ beginning at $\widehat{p}$ and  let $\widehat{R}$ be the component of $\pi^{-1}(R)$ 
containing $\widehat{p}$.
As $\pi_1(R,p)\leq G$, $\widehat{R}$ is a lift of $R$.

Suppose that $\widehat{\rho}(1)\in\widehat{R}$.
Then there is a path $\widehat{\rho}'$ in $\widehat{R}$ from $\widehat{\rho}(1)$ to $\widehat{p}$. 
This path projects to a loop $\rho'$ in $R$ based at $p$. 
Since $\widehat{\rho}\cdot\widehat{\rho}'$ is a loop in $\widehat{M}$, $\rho\cdot\rho'\in G$.
However, since $\rho'\in G$, this contradicts that $\rho\notin G$.
Hence $\widehat{\rho}(1)\notin\widehat{R}$.
Let $\widehat{R}'$ be the component of $\pi^{-1}(R)$ containing $\widehat{\rho}(1)$.
As $\rho\in\ker(\Phi_M)$, there is an isotopy $H\colon M\times[0,1]\to M$ from the identity to itself that can be used to define $\phi_{\rho}$.
This lifts to an isotopy $\widehat{H}\colon\widehat{M}\times[0,1]\to\widehat{M}$ that begins at the identity.
This isotopy takes $\widehat{R}$ to $\widehat{R}'$. Therefore $\widehat{R}'$ is a lift of $R$.

Choose $\varepsilon>0$ such that $(\rho(0,\varepsilon)\cup\rho(1-\varepsilon,1))\cap R=\emptyset$.
Remark \ref{irreducibleremark} tells us that $\widehat{M}$ is irreducible, since $M$ is.
By Lemma \ref{pi1injectivelemma}, $R$ is $\pi_1$--injective in $M$.
As $\pi_1(R,p)\leq G$, we see that $\widehat{R}$, and so also $\widehat{R}'$, is $\pi_1$--injective and hence incompressible.
By Proposition \ref{productregionprop}, $\widehat{R}$ and $\widehat{R}'$ bound a manifold $N$ homeomorphic to $R\times[0,1]$.
Let $\widehat{H'}\colon \widehat{M}\times[0,1]\to \widehat{M}$ be an isotopy that pushes $\widehat{R}'$ across $N$ to $\widehat{R}$.
Then the isotopy given by doing first $\widehat{H}$ and then $\widehat{H}'$ ends at a map $\widehat{\psi}\colon\widehat{M}\to\widehat{M}$ with $\widehat{\psi}(\widehat{R})=\widehat{R}$.
By Lemma \ref{swapsideslemma}, $\psi$ does not switch the sides of $\widehat{R}$.
This tells us that $\widehat{\rho}(0,\varepsilon)$ and $\widehat{\rho}(1-\varepsilon,1)$ either both lie in $N$ or both lie outside of $N$, as $\rho$ crosses $R$ transversely at $p$.
From this we see that $M$ is a quotient of a manifold obtained from $N$ by gluing $R\times\{0\}$ to $R\times\{1\}$ by some homeomorphism.
\end{proof}

\begin{theorem}[\cite{MR0415619} Theorem 10.6]\label{finiteindexthm}
Let $M'$ be a closed, orientable $3$--manifold. Then there does not exists a finite index subgroup $G$ of $\pi_1(M')$ such that $G$ is isomorphic to the fundamental group of a closed surface other than $\sphere$ or $\mathbb{R}P^2$.
\end{theorem}

\begin{theorem}[\cite{MR0415619} Theorem 11.1]\label{infindexthm}
Let $M'$ be a closed, orientable $3$--manifold. Suppose $G$ is a finitely generated normal subgroup of $\pi_1(M')$ of infinite index, with $G\ncong\mathbb{Z}$. Then there is a closed surface $R$ in $M'$ such that $G$ is a finite index subgroup of $\pi_1(R)$ and either $M'$ is an $R$--bundle over $\crcle$ or $M'$ is the union of two twisted $[0,1]$--bundles whose intersection, $R$, is the $\{0,1\}$--bundle of each.
\end{theorem}

\begin{proposition}\label{torusprop}
Suppose $M$ is a closed, orientable, irreducible $3$--manifold containing an embedded, incompressible torus $R$.
Then one of the following holds.
\begin{itemize}
\item $\ker(\Phi)\cong 1$ and there is an exact sequence
\[1\to\pi_1(M,p)\to\Mod(M_p)\to \Mod(M)\to 1.\]
\item $M$ is finitely covered by an $R$--bundle over $\crcle$.
\item $M$ is Seifert fibred and $\ker(\Phi)\cong\mathbb{Z}$ is generated by a power of a regular fibre.
\end{itemize}
\end{proposition}
\begin{proof}
Suppose $M$ is not finitely covered by an $R$--bundle over $\crcle$.
Then, by Lemma \ref{hakenlemma}, $\ker(\Phi)\trianglelefteq\pi_1(R)\cong \mathbb{Z}^2$.
In particular, $\ker(\Phi)$ is finitely generated.
Assume $\ker(\Phi)\ncong 1$.

If $\ker(\Phi)\cong\mathbb{Z}$ then, by Proposition \ref{cyclicsubgroupprop}, $M$ is Seifert fibred and $\ker(\Phi)$ is a subgroup of the infinite cyclic group generated by a regular fibre.

Now assume that $\ker(\Phi)\cong\mathbb{Z}^2$.
Suppose $\ker(\Phi)$ has finite index in $\pi_1(M)$.
Then $\pi_1(R)$ also has finite index in $\pi_1(M)$.
By Theorem \ref{finiteindexthm}, this cannot be the case. 
Hence $\ker(\Phi)$ has infinite index in $\pi_1(M)$.
Therefore Theorem \ref{infindexthm} applies. Let $R'$ be the surface given by Theorem \ref{infindexthm}.
As $R'$ is $2$--sided in $M$, it must be orientable. Given that $\ker(\Phi)$ is a finite index subgroup of $\pi_1(M)$, we see that $R'$ is a torus.
We have assumed that $M$ is not fibred with fibre a torus, so $M$ is formed from two twisted $[0,1]$--bundles over a third surface $R''$. 
Again using that $M$ is orientable, we find that $R''$ is non-orientable. Since it is double-covered by $R'$, this implies that $R''$ is a Klein bottle. Then $M$ is also fibred over $\crcle$ with fibre a torus (see \cite{MR565450} VI.15c).
\end{proof}

\begin{theorem}[\cite{MR0377891} Theorem 1]\label{coversfmfldthm}
A closed, orientable, irreducible Seifert fibred $3$--manifold with infinite fundametal group is finitely covered by an $\crcle$--bundle over a closed, orientable surface with genus at least $1$.
\end{theorem}

\begin{lemma}\label{threefibreslemma}
Suppose that $M$ is an orientable Seifert fibred $3$--manifold with base space a $2$--sphere and three exceptional fibres.
Then either $\pi_1(M)$ is finite or $M$ is finitely covered by a manifold satisfying Proposition \ref{torusprop}.
\end{lemma}
\begin{proof}
Assume that $\pi_1(M)$ is not finite. Note that $M$ is irreducible (see for example \cite{MR565450} VI.7).
By Theorem \ref{coversfmfldthm}, $M$ is finitely covered by a manifold $\widehat{M}$ that is an $\crcle$--bundle over a closed, orientable surface $N$ with genus at least $1$. That is, $\widehat{M}$ is Seifert fibred over $N$ with no exceptional fibres.
Taking the $\crcle$--bundle over an essential simple closed curve in $N$ gives an incompressible torus in $\widehat{M}$. 
In addition, $\widehat{M}$ is closed, orientable and irreducible.
Thus it satisfies Proposition \ref{torusprop}.
\end{proof}

\section{Conclusions}\label{proofsection}
%master document is splitlinks.tex

\begin{theorem}[\cite{MR1189862} Corollaries 8.3, 8.6]\label{torusthm}
Let $M'$ be a compact, orientable, irreducible $3$--manifold with infinite fundamental group.
Then $M$ is Seifert fibred if and only if $\pi_1(M')$ contains a cyclic normal subgroup.
If there exists $G\leq\pi_1(M')$ such that $G\cong\mathbb{Z}^2$,
then $M'$ either is Seifert fibred or contains an embedded, $\pi_1$--injective torus.
\end{theorem}
%The proof of this is not in the PL category. We may assume though that M' is SF in TOP. Being SF is defined in terms of fibres having neighbourhoods homeomorphic to specific other 3-manifolds. By Moise (eg Geom Top Dim 2,3 Thm35.3), a TOP 3-manifold is also a PL manifold. Two PL 3-manifolds that are TOP homeo are also PL homeo (Thm36.2). So M' meets the defn of SF in PL.
%Suppose instead there is a TOP map of an essential torus. Cut, isotope the gluing map to a PL map, form the new manifold (which has PL torus), see there is a TOP homeo between them, then there is a PL map between them. See below.

\kernelthm
\begin{proof}
We will show that $M$ is finitely covered by a manifold $M''$ for which $\ker(\Phi_{M''})$ is isomorphic to a subgroup of $\mathbb{Z}^3$. 
Then Theorem \ref{kernelthm} will follow from Lemma \ref{coverlemma}.
With this aim, we may take finite covers of $M$ as necessary. In particular, we may assume that $M$ is orientable.

If $M$ has a boundary component that is not a torus then Proposition \ref{nottorusprop} shows that $\ker(\Phi_M)\cong 1$.
Suppose that every component of $\partial M$ is a torus.
If $M$ is not prime, it contains a separating $2$--sphere. Then Proposition \ref{nonprimeprop} and Lemma \ref{rp3lemma} together imply that $\ker(\Phi_M)\cong 1$.
If $M$ is prime but not irreducible then $M\cong\sphere\times\crcle$ (see, for example, \cite{MR0415619} Lemma 3.3), and Corollary \ref{productscor} gives that $\ker(\Phi_M)=\pi_1(M,p)\cong\mathbb{Z}$.
Therefore we may now assume that $M$ is irreducible.
If $M$ has compressible boundary, then $M$ is a solid torus, as noted in Remark \ref{solidtorusremark}. Corollary \ref{productscor} again says that in this case $\ker(\Phi_M)=\pi_1(M,p)\cong\mathbb{Z}$.
Hence we may also assume that $M$ is $\partial$--irreducible.

Suppose first that $M$ is not fibred over $\crcle$.
If $\partial M\neq\emptyset$  and $\ker(\Phi_M)\ncong 1$ then, by Proposition \ref{alltoriprop}, $M$ is Seifert fibred and $\ker(\Phi_M)\cong\mathbb{Z}$.
Suppose instead that $M$ is closed.
If $\pi_1(M,p)$ is finite, by the Elliptization Theorem, $M$ is a spherical manifold and is finitely covered by $\Sphere$ (see \cite{2012arXiv1205.0202A} Theorem 1.12). Note that if $M\cong\Sphere$ then $\ker(\Phi_M)=\pi_1(M,p)\cong 1$.
Further suppose, therefore, that $\pi_1(M,p)$ is infinite.

If $M$ contains an incompressible torus and $\ker(\Phi_M)\cong 1$ then, by Proposition \ref{torusprop}, either $M$ is Seifert fibred with $\ker(\Phi_M)\cong \mathbb{Z}$, or $M$ has a finite cover that is a torus bundle over $\crcle$. 

Suppose $M$ does not contain an embedded incompressible torus, but there exists a map $f\colon\crcle\times\crcle\to M$ such that the induced map $f_*\colon\pi_1(\crcle\times\crcle)\to\pi_1(M)$ is injective.
By Theorem \ref{torusthm}, $M$ is then Seifert fibred. 
From the absence of any embedded incompressible tori in $M$, it follows that $M$ is Seifert fibred over a sphere with at most three exceptional fibres or over $\mathbb{R}P^2$ with at most one exceptional fibre (see, for example, \cite{MR565450} VI.7).
On the other hand, the existence of such a map $f$ implies that the surface is a sphere and there must be at least three exceptional fibres.
Since we have assumed that $\pi_1(M,p)$ is infinite, Lemma \ref{threefibreslemma} and Proposition \ref{torusprop} together give that $M$ is finitely covered by a manifold $M'$ such that either $\ker(\Phi_{M'})\leq \mathbb{Z}$ or $M'$ is a torus bundle over $\crcle$.

Hence we can further assume that $M$ is atoroidal.
The Hyperbolization Theorem says that $M$ is hyperbolic with finite volume (see \cite{2012arXiv1205.0202A} Theorem 1.13).  
By the Virtual Fibering Conjecture (\cite{2012arXiv1204.2810A} Theorem 9.2; see also \cite{2012arXiv1210.4799F}), $M$ is therefore finitely covered by a manifold fibred over $\crcle$.
%It is not immediately clear that the proof of the Virtual Fibering Conjecture holds in the PL category. In Agol, A criterion for virtual fibering, it seems that the PL category is used.
%We may assume though that it holds in TOP. If M' is fibred, there is a surface N', and M' cut along N' is TOP homeo to N' x [0,1]. By Epstein A.4, the gluing map is isotopic to a PL homeo. Now the two glued manifolds are TOP homeo, so are also PL homeo. Hence M' is PL fibred.
%In addition, since covering is define by local homeo, TOP cover is the same as PL cover.

It only remains to consider the case that $M$ is fibred over $\crcle$.
Denote by $N$ the base space, and $\omega\colon (N,\partial N)\to (N,\partial N)$ the monodromy.
If $\omega^m$ is not isotopic to the identity on $N$ for any $m\in\mathbb{Z}\setminus\{0\}$ then, by Corollary \ref{notperiodiccor}, $\ker(\Phi_M)$ is isomorphic to a subgroup of $\mathbb{Z}^2$ (in fact $\ker(\Phi_M)\cong 1$ unless $N$ is a torus).
Suppose instead that there exists $m\in\mathbb{Z}\setminus\{0\}$ such that $\omega^m$ is isotopic to the identity on $N$. 
Then $M$ has a $|m|$--fold cover that is fibred over $N$ with monodromy $\omega^m$.
We may therefore assume that in fact $\omega$ is isotopic to the identity.
By Proposition \ref{productbundleprop}, it then follows that $\ker(\Phi_M)\cong\ker(\Phi_N)\times\mathbb{Z}$.
If $\chi(N)\neq 0$ then Proposition \ref{surfacesprop} shows that $\ker(\Phi_N)\cong 1$.
Corollary \ref{chi0cor} gives that if $\chi(N)=0$ then either $\ker(\Phi_N)\cong\mathbb{Z}$ or $\ker(\Phi_N)\cong\mathbb{Z}^2$.
\end{proof}

\begin{theorem}[\cite{MR0415619} Corollary 9.9]\label{torsionfreethm}
Let $M'$ be a prime, orientable $3$--manifold with $\pi_1(M')$ infinite. 
Then $\pi_1(M')$ contains no finite order elements.
\end{theorem}

\begin{theorem}[\cite{MR565450} VI.29]\label{sfcoverthm}
Let $M$ be a compact, connected, orientable, Haken $3$--manifold, and let $\widehat{M}$ be a finite cover of $M$. Then $M$ is a Seifert fibred manifold if and only if $\widehat{M}$ is.
\end{theorem}

\hypcor
\begin{proof}
Assume that the given exact sequence does not hold. That is, assume that $\ker(\Phi_M)\ncong 1$.
From the proof of Theorem \ref{kernelthm}, we know that $M$ is prime, all boundary components of $M$ are tori, and one of the following holds.
\begin{itemize}
\item[(a)] $M\cong\sphere\times\crcle$.
\item[(b)]  $M$ is a solid torus.
\item[(c)]  $M$ is Seifert fibred with non-empty boundary.
\item[(d)]  $M$ is a closed, spherical manifold.
\item[(e)]  $M$ is closed, contains an incompressible torus, and is Seifert fibred.
\item[(f)] $M$ is irreducible and closed, contains an incompressible torus, and is finitely covered by a torus bundle over $\crcle$.
\item[(g)]  $M$ is Seifert fibred over a sphere with three exceptional fibres.
\item[(h)]  $M$ is closed, hyperbolic, and finitely covered by a surface bundle over $\crcle$.
\item[(i)]  $M$ is a torus bundle over $\crcle$ with non-periodic monodromy.
\item[(j)]  $M$ is a surface bundle over $\crcle$ with periodic monodromy.
\end{itemize}
Spherical manifolds are Seifert fibred (see \cite{2012arXiv1205.0202A} p13 or \cite{MR0426001} Theorem 5). %see also Hatcher 3-manifolds survey
Therefore, if we further assume that $M$ is not Seifert fibred then we can rule out possibilities (a)--(e) and (g).

Suppose that case (h) holds. Since $M$ is hyperbolic with finite volume, every cover of $M$ is also hyperbolic with finite volume, and therefore is not Seifert fibred (see, for example, \cite{MR1435975} Theorems 4.7.8 and 4.7.10) and does not contain an essential torus. 
%see also Hatcher 3-manifolds survey; Scott Geometries of 3-manifolds Thm 5.2; Gordon--Wu Toroidal and annular Dehn fillings
Thus $M$ is finitely covered by an $N$--bundle $\widehat{M}$ for some closed surface $N$ other than a torus. If the monodromy is not periodic then Corollary \ref{notperiodiccor} shows that $\ker(\Phi_{\widehat{M}})\cong 1$, and so Lemma \ref{coverlemma} says that $\ker(\Phi_M)$ is finite.
However, by Theorem \ref{torsionfreethm} there are no finite non-trivial subgroups of $\pi_1(M)$.
Thus $\ker(\Phi_M)\cong 1$.
On the other hand, if the monodromy is periodic then $M$ is also finitely covered by $N\times\crcle$. This contradicts that no cover of $M$ is Seifert fibred, so this cannot occur.

Suppose next that case (j) holds.
Suppose that $M$ is an $N$--bundle for a compact surface $N$, with monodromy $\omega\colon (N,\partial N)\to (N,\partial N)$, and that $\omega^m$ is isotopic to the identity on $N$ for some $m\neq 0$.
Then $\pi_1(M)$ is an HNN extension of $\pi_1(N)$. By direct calculation we find that the $m$th power of the stable letter generates an infinite cyclic normal subgroup of $\pi_1(M)$. Theorem \ref{torusthm} shows that, again, $M$ is Seifert fibred.

Finally, suppose that case (f) holds. Let $\widehat{M}$ be a finite cover of $M$ that is a torus bundle over $\crcle$. 
From Corollary \ref{torusbundlecor} we see that either $\ker(\Phi_{\widehat{M}})\cong 1$, $\ker(\Phi_{\widehat{M}})\cong \mathbb{Z}$ or $\widehat{M}$ is the $3$--torus.
In the first case Lemma \ref{coverlemma} and Theorem \ref{torsionfreethm} again combine to show that $\ker(\Phi_M)\cong 1$, contradicting our opening assumption.
In the third case $\pi_1(\widehat{M})\cong\mathbb{Z}^3$.
Using Theorem \ref{torusthm}, we therefore find that $\widehat{M}$ is Seifert fibred. Theorem \ref{sfcoverthm} gives that $M$ is also Seifert fibred.
\end{proof}

%------------------

\bibliography{pointpushreferences}
\bibliographystyle{hplain}

%------------------

\bigskip
\noindent
CIRGET, D\'{e}partement de math\'{e}matiques

\noindent
UQAM

\noindent
Case postale 8888, Centre-ville

\noindent
Montr\'eal, H3C 3P8

\noindent
Quebec, Canada

\smallskip
\noindent
\textit{jessica.banks[at]lmh.oxon.org}

\end{document}